\documentclass[oneside]{amsart}

\usepackage[width=160mm,top=30mm,bottom=32mm]{geometry}
\usepackage{caption}
\usepackage{subcaption}
\usepackage[centertags]{amsmath}
\usepackage{bbm, amsfonts,amssymb, amsthm}
\usepackage{mathrsfs}
\usepackage{fancyhdr}
\usepackage{graphicx}
\usepackage{amsmath}
\usepackage{tikz-cd} 
\usepackage{physics}
\usepackage[sorting=nyt]{biblatex}
\usepackage{mathtools}
\usepackage{hyperref}
\usepackage{lipsum}

\newtheorem*{Theorem*}{Theorem}
\newtheorem{Problema}{Problem}

\newtheorem{TheoremA}{Theorem}

\newtheorem{TheoremB}{Theorem}

\newtheorem{TheoremC}{Theorem}

\newtheorem{TheoremD}{Theorem}

\newtheorem{thm}{Theorem}[section]

\newtheorem{lema}[thm]{Lemma}
\newtheorem{prop}[thm]{Proposition}
\newtheorem{cor}[thm]{Corollary}
\newtheorem{exam}[thm]{Example}
\newtheorem{defn}[thm]{Definition} 
\newtheorem{rem}[thm]{Remark}

\newcommand{\cC}{\mathcal{C}}
\newcommand{\cF}{\mathcal{F}}

\newcommand{\cJ}{\mathcal{J}}
\newcommand{\cL}{\mathcal{L}}

\newcommand{\cQ}{\mathcal{Q}}

\newcommand{\cU}{\mathcal{U}}
\newcommand{\cX}{\mathcal{X}}

\newcommand{\bB}{\mathbb{B}}
\newcommand{\bC}{\mathbb{C}}

\newcommand{\bN}{\mathbb{N}}

\newcommand{\bR}{\mathbb{R}}
\newcommand{\bS}{\mathbb{S}}

\newcommand{\bZ}{\mathbb{Z}}

\newcommand{\nc}{\newcommand}
\nc{\p}{\partial}
\nc{\ol}{\overline}

\def\@maketitle{%
  \normalfont\normalsize
  \let\@makefnmark\relax  \let\@thefnmark\relax
  \ifx\@empty\@date\else \@footnotetext{\@setdate}\fi
  \ifx\@empty\@subjclass\else \@footnotetext{\@setsubjclass}\fi
  \ifx\@empty\@keywords\else \@footnotetext{\@setkeywords}\fi
  \ifx\@empty\thankses\else \@footnotetext{%
    \def\par{\let\par\@par}\@setthanks}\fi
  \@mkboth{\@nx\shortauthors}{\@nx\shorttitle}%
  \global\topskip42\p@ 
  \@settitle
  \ifx\@empty\authors \else \@setauthors \fi
  \ifx\@empty\@commby
  \else
    \baselineskip18\p@
    \vtop{\centering{\footnotesize\@commby\@@par}%
      \global\dimen@i\prevdepth}\prevdepth\dimen@i
  \fi
  \ifx\@empty\@dedicatory
  \else
    \baselineskip18\p@
    \vtop{\centering{\footnotesize\itshape\@dedicatory\@@par}%
      \global\dimen@i\prevdepth}\prevdepth\dimen@i
  \fi
  \@setabstract
  \normalsize
  \dimen@34\p@ \advance\dimen@-\baselineskip
  \vskip\dimen@\relax
}

\definecolor{myteal}{HTML}{00797d}
\hypersetup{
    colorlinks=false,%
    citebordercolor=green,%
    filebordercolor=red,%
    linkbordercolor=blue,%
    urlbordercolor=red}

\title[Minimal surfaces with closed curvature lines]{Minimal surfaces with closed curvature lines}
\author{CARLOS ANDRÉS TORO CARDONA}
\address{Instituto de Matemática Pura e Aplicada-Instituto de Matemática e Estatística, Universidade Federal do Rio Grande do Sul, Brazil}
\email{carlos.toroc@impa.br}
\date{}

\begin{document}
\begin{abstract}
We investigate complete non-orientable minimal surfaces of finite total curvature in $\bR^3$ such that their ends are foliated by closed lines of curvature. This condition on the ends is necessary if they have a piece inside some Euclidean ball that is free boundary. It turns out this is a rigid situation, and we are able to show, among further obstructions, that there are no such surfaces with one end.
\end{abstract}

\maketitle

\section{Introduction}
This article is motivated by the hope of constructing a non-orientable free boundary minimal surface in the unit ball $\bB^3$. The method of construction that we propose is to obtain such a surface as the intersection of a complete non-orientable minimal surface of finite total curvature in the whole Euclidean space with a sphere $\bS_R^2$, of sufficiently large radius $R$ and arbitrary center, in an orthogonal way along the ends of the extended minimal surface. After an appropriate translation and a rescaling, the piece of the surface lying in the interior of the sphere $\bS^2_R$ would become a free boundary minimal surface in the unit ball $\bB^3.$

Such a method of construction turns out to be extremely rigid, even if we look at the question in the case of free-boundary orientable minimal surfaces. The plane and the critical catenoid seem to be the only known free boundary minimal surfaces in $\bB^3$ which are produced in this way, but so far, to the best of our knowledge, there is no proof of such assertion.

Suppose that $X:M\to \bR^3$ is a complete minimal immersion of finite total curvature, where $M=\Sigma\setminus \{p_1,\ldots p_N\}$ and $\Sigma$ is a compact surface without boundary. By a theorem due to Jorge-Meeks \cite[Theorem 1]{jorgemeeks}, for sufficiently large radius $R$, the set $X^{-1}(X(M)\cap\bS^2_R)$ is composed of a finite number of closed curves $\gamma_1,\ldots, \gamma_N$ circling the end points $p_1,\ldots, p_N$ respectively.

By the classical Joachimsthal's Theorem, if the intersection of $X(M)$ with $\bS^2_R$ is orthogonal, then each $\gamma_i$ is a line of curvature on $M$. Let us denote by $\Sigma$ the piece of the surface $M$ which lies in the interior of the corresponding ball of radius $R$, so that $\partial \Sigma=\bigcup_{i=1}^N\gamma_i$. Since the unit conormal vector $\nu_i$ of $\Sigma$ along $\gamma_i$ coincides with the unit normal vector to the sphere $\bS^2_R$, the geodesic curvature of $\gamma_i$ on $\Sigma$ is the same as the normal curvature of $\gamma_i$ on the sphere, which is constant and equal to $\frac{1}{R}$.

The converse of these statements was investigated by J. Lee and E. Yeon. As we will see, they provided, in particular, a clear method of construction of free boundary minimal surfaces as restrictions of complete minimal surfaces in $\bR^3$ of finite total curvature. We summarize their result as follows:
\begin{Theorem*}[{\textit{Cf. }J. Lee, E. Yeon \cite[Proposition 3.3]{JLee}}]
Let $M$ be a minimal surface in $\bR^3$ of finite total curvature. If the intersection of $M$ with a sphere $\bS^2_R$ is orthogonal along finitely many closed curves $\gamma_i$, then each $\gamma_i$ is a closed line of curvature of constant geodesic curvature $\frac{1}{R}$ on $M$. Conversely, if $\gamma_1,\ldots, \gamma_N$ is a collection of closed lines of curvature with non-zero geodesic curvature $c_i$ on $M$, then $M$ intersects orthogonally a collection of spheres $\bigcup_{i=1}^N\bS^2_{R_i}$ of radii $R_i=\frac{1}{\abs{c_i}}$ and possibly different centers.
\end{Theorem*}

We study the class $\cC$ of complete minimal surfaces of finite total curvature in $\bR^3$ with \textit{closed curvature lines} around the ends, see Definition \ref{classC}, in order to investigate the following:
\begin{Problema}\label{problemrestriction}
    What are the free boundary minimal surfaces in the unit ball $\bB^3$ which are the restriction of a complete minimal surface in $\bR^3$ of finite total curvature?
\end{Problema}

Indeed, one would need to study, inside the class $\cC$, the minimal surfaces $\Sigma$ whose closed lines of curvature $\gamma_i$ all have the same constant geodesic curvature, and then be able to control the centers of the associated spheres $\bS^2_i$ intersecting orthogonally the surface $\Sigma$ along $\gamma_i$, in order to check that they all coincide.

We determine obstructions theorems and topological conditions for the class $\cC$. For minimal surfaces with a single end we found:
\begin{TheoremA}\label{nonexistenceoneendanygenus}
\textit{There is no complete single-ended minimal surface with finite total curvature in $\bR^3$ which contains a closed curvature line around its end.}
\end{TheoremA}

As a consequence, none of the orientable free boundary minimal surfaces with one boundary component and any genus discovered by M. Schulz, G. Franz and A. Carlotto can be obtained as a restriction of the Chen-Gackstatter surface or its generalizations to higher genus. Furthermore, in the non-orientable context, the Klein bottle of F. Lopez \cite{Lopez} and none of its generalizations to higher genus found by F. Martin \cite{MartinSurfaces} with a single end can intersect a sufficiently large sphere in an orthogonal way.

Similarly, we obtain a further restriction:
\begin{TheoremB}\label{obstructionplanar}
   The embedded planar ends of a complete minimal immersions in $\bR^3$ of finite total curvature do not admit a closed line of curvature.
\end{TheoremB}

An application of Theorem \ref{obstructionplanar} is that the Costa surface \cite{CostaArticle} and the corresponding Costa-Hoffman-Meeks family \cite{CostaHigherGenus} of higher genus cannot intersect a ball of sufficiently large radius in an orthogonal way. Even though the Kaopuleas-Li \cite{KAPOULEASLI} and Ketover \cite{KETOVERG0} free boundary minimal surfaces are the free boundary analogues of the Costa-Hoffmann-Meeks surfaces for sufficiently large genus, it is interesting to remark that it is not known if there exists a free boundary minimal analog in the unit ball of the Costa's surface.

While Theorem \ref{nonexistenceoneendanygenus} contains our most definite conclusion about this method of construction, we are able to obtain several partial results in the case of surfaces with more than one end. First we constructed, to the best of our knowledge, a new deformation family of a classical surface due to Oliveira \cite{Oliveira}, with the topology of a twice punctured projective space, having a closed curvature line around exactly one of its ends:
\begin{TheoremC}\label{secondsurface}
    \textit{There exists a one-parameter family of complete minimal immersions of a projective space with two ends in $\bR^3$ and total curvature $-10\pi$. This family have only one catenoidal end, and moreover, it is foliated by closed lines of curvature.}
\end{TheoremC}

We were also able to provide the following obstruction theorem in the case of a twice punctured projective space with \textit{parallel ends}:
\begin{TheoremD}\label{theorem6}
   \textit{There does not exist a complete minimal immersion of a projective space with two parallel ends and finite total curvature with closed curvature lines around both ends.}
\end{TheoremD}

By the non-existence of free boundary minimal Möbius bands in $\bB^3$ \cite{Carlos}, the simplest model of a non-orientable minimal surface in $\bR^3$ with closed curvature lines around their ends is the minimal immersion of a twice-punctured real projective space $\mathbb{RP}^2$, or more generally, $n$-punctured real projective spaces.

We found in the literature a family of minimal projective spaces in $\bR^3$ with an even number $N\geq 4$ of ends, due to S. Kato and K. Hamada \cite{Kato}, which are the non-orientable analogs of the Jorge-Meeks nodoids. For appropriate parameters, some members of this family are surfaces with closed curvature lines around their ends. We think that these surfaces are promising prototypes of non-orientable free boundary minimal surfaces in the unit ball $\bB^3$.

This article is organized as follows. In Section \ref{Section 4.2}, we review some preliminary results from the theory of Riemann surfaces.

In Section \ref{Section 4.3}, we focus on complete orientable minimal surfaces in $\bR^3$ of finite total curvature. We study the qualitative behavior of the maximal and minimal principal foliation around the ends and umbilical points of such immersions, and introduce the class $\cC$. We prove a useful characterization result and a topological restriction for the elements of $\cC$, see Corollary \ref{CriterioclassC} and Corollary \ref{ClosedCurvatureLineFormula}. We also prove Theorem \ref{obstructionplanar}, see Theorem \ref{PlanarEndsDoNotAdmitClosedLines}.

In Section \ref{Section 4.4}, we make a survey of the theory of complete non-orientable minimal surfaces in $\bR^3$ and the main constructions found in the literature. We prove in Corollary \ref{closedlineformulanonoriented} a generalization of the topological restriction for the existence of a non-orientable minimal surface in the class $\cC$. Then we found the transformational law of the Hopf differential associated to a complete non-orientable minimal surface in $\bR^3$ (see Proposition \ref{transformationHopfgeneralized}) and use it to prove Theorem \ref{nonexistenceoneendanygenus}, see Theorem \ref{Nonexistencetheorem}. Finally, we obtain some partial results for the possibility of finding an element in $\cC$ with the topology of a projective space with two or more ends in $\bR^3$. We construct the surface of Theorem \ref{secondsurface} (see Theorem \ref{CarlosTSecond}) and prove the non-existence Theorem \ref{theorem6}, see Theorem \ref{Nonexistencetwopuncturesclosed}.\\\\ \textbf{Acknowledgements.} I am deeply grateful to my PhD advisor Lucas Ambrozio for his constant encouragement. I thank professor R. Garcia for the discussions about the dynamics of the lines of curvature and professor M. Weber for his help using the Mathematica software. I thank my colleagues I. Miranda and C. Palacio for the inspiring conversations that motivated many ideas in this article. I am also grateful to professor Vanderson Lima for his improvements on preliminary versions of this work. Finally, I thank FAPERJ (Grant number E26/202.321/2024) for supporting this research project.
\section{Preliminaries}\label{Section 4.2}
We start by recalling known results from the theory of Riemann surfaces.
\begin{defn}\label{multiplicity}
\normalfont{
   Let $F:M\to N$ be a holomorphic map between two Riemann surfaces. Consider charts $(U,\psi)$ of $M$ and $(V,\phi)$ of $N$ such that $\psi(0)=P\in M$ and $\phi(0)=F(P)\in N$, which we call centered at $P$ and $F(P)$ respectively. Consider the holomorphic map
   \begin{equation*}
       \hat{F}\coloneqq \phi^{-1}\circ F\circ \psi:U\subset \bC\to V\subset \bC.
   \end{equation*} 
   Since $\hat{F}(0)=0$, there exists a natural number $n\in \bN$ and a holomorphic function $H(z)$ with $H(0)\neq 0$ such that
\begin{equation*}
   \hat{F}(z)=z^nH(z).
\end{equation*}
The integer $n$ does not depend on the coordinate charts and it is called \textit{the multiplicity} $mult_{P}(F)$ of the function $F$ at the point $P$. We define \textit{the ramification order} of $F$ at $P$ as $r_{P}(F)=mult_{P}(F)-1$. If $r_{P}(F)\geq 1$ then $P$ is called a \textit{branch point} of $F$, and its image $F(P)$ a \textit{ramification point}.
}
\end{defn}

Another useful definition is the following:
\begin{defn}
\normalfont{
    Let $F:\Sigma\to \bS^2\simeq \bC \cup \{\infty\}$ be a holomorphic function. The order of $F$ at point $q\in \Sigma$ is defined as
    \begin{gather*}
         ord_q(F)\coloneqq  \begin{cases}
        \hspace{0.25cm} mult_q(F)\quad \text{if $q$ is a zero of $F$},\\  -mult_q(F)\hspace{0.35cm} \text{if $q$ is a pole of $F$},\\ \hspace{0.8cm}0\hspace{1.25cm} \text{otherwise}.
        \end{cases}
    \end{gather*}
    }
\end{defn}
\begin{lema}\label{function in terms of order}
    Let $F:\Sigma\to \bS^2$ be a holomorphic function. Let $\psi:U\subset \bC\to \Sigma$ be any chart around $q\in \Sigma$. Then there exists a holomorphic function $H$ with $H(0)\neq 0$ such that
    \begin{equation*}
        F(\psi(z))=z^{ord_q(F)}H(z).
    \end{equation*}
\end{lema}

\begin{lema}\label{BranchPointCriteria}
    Let $F:M\to N$ be a holomorphic map. Then $P\in M$ is a branch point of $F$ if and only if $dF(P)=0$. Moreover 
    \begin{equation*}
        r_P(F)=mult_{P}(dF).
    \end{equation*}
\end{lema}

\begin{prop}\label{dichotomy}
   Let $F:\Sigma\to \bS^2$ be a holomorphic map. For any $q\in \Sigma$ we have the dichotomy
    \begin{enumerate}
        \item If $ord_q(F)\neq 0$ then the one-form $\frac{dF}{F}$ has a simple pole at $q$.
        \item If $ord_q(F)=0$ then the one-form $\frac{dF}{F}$ has a zero at $q$ of multiplicity equal to $r_q(F).$ 
    \end{enumerate}
\end{prop}
\begin{lema}\label{Chartrelatedpoint}
    Suppose that $\Sigma$ is a Riemann surface with an antiholomorphic involution $\tau:\Sigma\to \Sigma$. Consider the conjugation map, $T:\bS^2\to \bS^2$ $T(z)=\overline{z}$.
   If $(D_{\epsilon},\phi)$ is a complex chart centered at $p\in \Sigma$, then the chart $(D_{\epsilon},\hat{\phi})$, given by
    \begin{equation*}
       \hat{\phi}\coloneqq \tau\circ \phi\circ T,
    \end{equation*}
is centered at $\tau(p)$ and it is compatible with the complex structure of $\Sigma$.
\end{lema}

\begin{lema}\label{Residuerelatedpoints}
    Let $M$ be a Riemann surface and $\tau$ an antiholomorphic involution on $M$. Then for any meromorphic one-form $\omega$ in $M$ and any point $P\in M$,
    \begin{equation*}
        \eval{Res}_{P}\overline{\tau^{*}\omega}=\overline{\eval{Res}_{\tau(P)}\omega}.
    \end{equation*}
\end{lema}
\section{Orientable minimal surfaces with finite total curvature}\label{Section 4.3}

The classical global version of the Weierstrass representation theorem for orientable minimal surfaces is: 
\begin{thm}[{Weierstrass representation \cite{MartinThesis}}]\label{globalWeierstrassOriented}
    Let $M$ be a Riemann surface, and let $\eta$ and $g$ be a $1$-holomorphic form and a meromorphic function on $M$, respectively. Define
    \begin{gather}\label{PhiWeierstrass}
        \phi_1=\frac{1}{2}(1-g^2)\eta,\quad \phi_2=\frac{i}{2}(1+g^2)\eta,\quad \phi_3=g\eta.
    \end{gather}
   If $\Phi=(\phi_1,\phi_2,\phi_3)$ has no real periods on $M$, that is,
    \begin{equation*}
        Re\left(\int_{\gamma}\Phi\right)=(0,0,0) \quad \forall [\gamma]\in H_1(M,\bZ),
    \end{equation*}
    and if moreover
    \begin{equation}\label{immersion condition}
    \sum_{j=1}^3|\phi_j|^2\neq 0,    
    \end{equation}
    then the map $X:M\to\bR^3$
    \begin{equation*}
    X(P)=Re\left(\int_{P_0}^P\Phi\right)
    \end{equation*}
    is a conformal minimal immersion.
\end{thm}
\begin{rem}\label{equivalent immersion condition}
\normalfont{
    The condition $\sum_{j=1}^3|\phi_j|^2\neq 0$ is equivalent to the requirement that the zeros of the one-form $\eta$ are exactly over the poles of $g$ with double order, indeed it holds that
   \begin{equation}\label{conformal factor immersion}
\sum_{j=1}^3|\phi_j|^2=\frac{1}{2}(1+|g|^2)^2|f|^2.
    \end{equation} 
    }
\end{rem}

The behavior of embedded ends is well understood by the following result: 
\begin{thm}[{\cite{HildebrandtI}}]\label{embeddedends}
    Let $X:\Sigma\setminus \{p_1,\ldots p_k\}\to \bR^3$ be a complete minimal surface of finite total curvature and without branch points. Let $E_j$ be an embedded end corresponding to the puncture $p_j$ and suppose that $N(p_j)=(0,0,1)$. Then, outside of a compact set, the end $E_j$ has the asymptotic behavior
    \begin{equation}\label{asymptoticend}
        z(x,y)=\alpha \log r+\beta+r^{-2}(\gamma_1x+\gamma_2y)+O(r^{-2}),\quad  r=\sqrt{x^2+y^2}\rightarrow \infty.
    \end{equation}
    We call the embedded end $E_j$  planar if $\alpha=0$ whereas, for $\alpha\neq 0$, we speak of a catenoidal end.
\end{thm}
\begin{rem}\label{catenoidalvsplanar}
\normalfont{
   The residue of the height differential $\phi_3$ of the Weierstrass data \eqref{PhiWeierstrass} determines whether the end is planar (zero residue)  or catenoidal (non-zero residue) \cite{CarlosThesis}.
   }
\end{rem}

\subsection{The dynamics of the lines of curvature}

The equatorial disk and the critical catenoid are solutions to Problem \ref{problemrestriction}, and we may wonder if they are the only free boundary minimal surfaces with this property.

Notice that if a complete minimal surface of finite total curvature in $\bR^3$, immersed or embedded, has a free boundary piece in some Euclidean ball, with the topology of a disk, then by Nitsche's uniqueness Theorem \cite{Nitsche1985}, it contains a totally geodesic slice, and therefore the whole complete minimal surface must be a plane.

If we restrict our attention to the case of one boundary component and non-zero genus, we may hope to find a solution to Problem \ref{problemrestriction} among the classical Chen-Gackstatter \cite{CG} surfaces with genus one and two of total curvatures $-8\pi$ and $-12\pi$ respectively, or among their generalizations to higher genus which were accomplished by the subsequent works of N. Do Espirito-Santo \cite{EspirituSanto}, C. Thayer \cite{Thyer}, K. Sato \cite{Sato}, M. Weber and M. Wolf \cite{WW}. However, we will prove that none of them restricts to a free boundary minimal surface in the ball, see Theorem \ref{Nonexistencetheorem}.    

In order to study Problem \ref{problemrestriction}, we apply Sotomayor and Gutierrez analysis of the asymptotic behavior of the Hopf differential around umbilical points and ends, in the particular case of a complete orientable minimal immersion $X:M\to \bR^3$ of finite total curvature. Let $(U_{\alpha},\psi_{\alpha})$ be an isothermal atlas of $M$ with $z_{\alpha}\in U_{\alpha}$, and consider the Hopf differential $Q_H$ on $M$ such that $\psi_{\alpha}^{*}Q_H=\phi_{\alpha}(z)dz_{\alpha}^2$. We refer to $\phi(z)$ as \textit{the associated complex function}, which depends upon the choice of chart.

By Osserman's Theorem $M$ is conformally diffeomorphic to a finitely punctured closed Riemann surface $M=\Sigma\setminus \{p_1,\ldots p_N\}$ and moreover the Weierstrass data $(g,\eta)$ of the immersion extends meromorphically to the whole $\Sigma$ \cite[Theorem 3.1]{OssermanAnnals}. This implies that the Hopf differential $Q_H$ extends also meromorphically to $\Sigma$ and we can prove that in fact it is free of poles in $M$:
\begin{prop}\label{Hopf free poles interior}
    The Hopf differential extends meromorphically to $\Sigma$, and it does not have any poles in $M=\Sigma\setminus \{p_1,\ldots, p_N\}.$
\end{prop}

While the Hopf differential is free of poles on $M$, it may have zeros which characterize exactly the set of umbilical points $\cU_X$ of the immersion $X$. We have the following equivalence

\begin{prop}\label{Hopf zeros ramification Gauss}
   Let $X:M\to \bR^3$ be a complete orientable minimal immersion of finite total curvature with $M=\Sigma\setminus\{p_1,\ldots, p_N\}$ and $\Sigma$ closed. Then, the following assertions are equivalent:
   \begin{enumerate}
   \item $q\in M$ is an umbilical point.
       \item $q\in M$ is a branch point of $g$.
       \item $q\in M$ is a zero of the Hopf differential $Q_H=\phi(z)dz^2$.
   \end{enumerate}
   Moreover the ramification order of $g$ at $q$ coincides with the multiplicity of the zero of the Hopf differential at $q$.
\end{prop}

\begin{defn}\label{maximalfoliation}
\normalfont{
   \textit{The maximal principal foliation} $\cF_{X}$ and \textit{the minimal principal foliation} $f_{X}$ are the integral lines of the smooth line fields $\cL_{X}$ and $l_{X}$ which are generated by the maximal and minimal principal directions, respectively.
    }
\end{defn}

We introduce the following definition, which is our object of study:
\begin{defn}\label{classC}
    \normalfont{Let $X:\Sigma\setminus\{p_1,\ldots, p_N\}\to \bR^3$ be a complete minimal immersion of finite total curvature, where $\Sigma$ is a closed Riemann surface. We say that $X$ has \textit{closed curvature lines around its ends} if for each $1\leq j\leq N$ either the maximal $\cF_X$ or the minimal $f_X$ principal foliation has a closed curvature line in a sufficiently small neighborhood $V_j$ of $p_j.$ We denote by $\cC$ the class of complete minimal surfaces of finite total curvature in $\bR^3$ with closed curvature lines around its ends.}
\end{defn}
\begin{exam}[{Jorge-Meeks nodoids}]
\normalfont{
The Weierstrass representation of these surfaces is defined on the $n+1$-punctured sphere
\begin{equation*}
     M=\bS^2\setminus \{z^{n+1}=1\},\quad  g(z)=z^n,\quad f(z)dz=\frac{dz}{(z^{n+1}-1)^2}.
\end{equation*}
 These orientable surfaces have embedded ends that are of catenoidal type, and their curvature lines are closed around each end, hence they belong to the class $\cC$ in the orientable context. The dynamics of the curvature lines is represented in Figure \ref{Jorge Meeks curvature lines}. The free boundary minimal analogues in the unit ball $\bB^3$ were discovered by Fraser and Schoen \cite{FS16}, but we do not know if they are pieces of the Jorge-Meeks nodoids.
}
\end{exam}
\begin{rem}
\normalfont{
    The same definition works for a non-orientable minimal surface of finite total curvature $X':M'\to \bR^3$ by considering the induced immersion $X:M\to \bR^3$ associated to the orientable double cover $M$ of $M'$.}
\end{rem}

The qualitative behavior of the principal foliation around the ends of the immersion $X$ is described by the following results, due to Gutierrez and Sotomayor:
\begin{figure}
     \centering
     \begin{subfigure}[b]{0.25\textwidth}
         \centering
         \includegraphics[width=\textwidth]{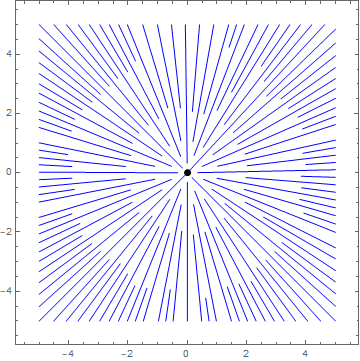}
         \caption{$a>0,\ b=0$.}
         \label{hopf n=-2a1b0}
     \end{subfigure}
     \hfill
     \begin{subfigure}[b]{0.25\textwidth}
         \centering
         \includegraphics[width=\textwidth]{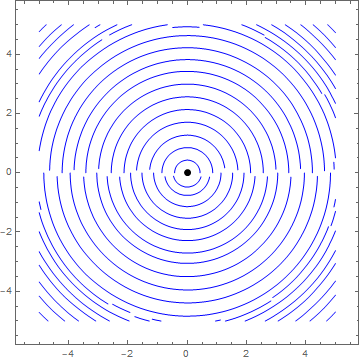}
         \caption{$a<0,\ b=0$.}
         \label{hopf n=-2a-1b0}
     \end{subfigure}
      \hfill
     \begin{subfigure}[b]{0.25\textwidth}
         \centering
         \includegraphics[width=\textwidth]{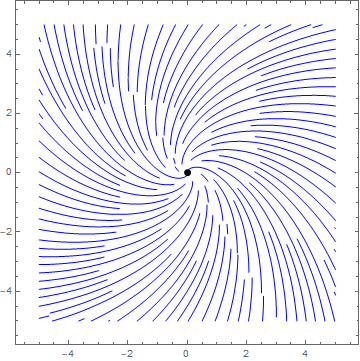}
         \caption{$b\neq 0$.}
         \label{hopf n=-2a1b2}
     \end{subfigure}
     \caption{Maximal principal foliation $F_X$ of Theorem \ref{dynamicslinescurvature} Item (b), with $n=2$.}
     \end{figure}

\begin{defn}\label{quadraticlimit}
    \normalfont{Let $X:\Sigma\setminus \{p_1,\ldots p_N\}\to \bR^3$ be a complete minimal immersion  of finite total curvature, such that its Hopf differential has a pole of order at most $2$ at each $p_j$. Let $\phi(z)$, $z=u+iv$, be the associated complex function to the Hopf differential with respect to isothermic coordinates of $\Sigma$ centered at the end $p_j$. We define the \textit{quadratic limit of the Hopf differential $Q_H$} at $p_j$ as the limit
    \begin{equation*}
        a+ib\coloneq \lim_{z\to0}z^2\phi(z)\in \bC.
    \end{equation*}}
\end{defn}
\begin{rem}
\normalfont{
    We point out that when the Hopf differential has a pole of order two at an end $p_j$, the quadratic limit of the Hopf differential does not depend on the choice of isothermic chart \cite{CarlosThesis}.}
\end{rem}

We propose the following slight modification of a result of Sotomayor and Gutierrez.
\begin{thm}[{\textit{Cf. }Corollary 4.7 \cite{SotoPrincipalLinesCMC}}]\label{dynamicslinescurvature}
   The lines of curvature of a complete minimal immersion $X:\Sigma\setminus \{p_1,\ldots, p_N\}\to \bR^3$ of finite total curvature around an end $p_j$ where the associated complex function $\phi(z)$ has a pole of order $n$ are described as follows:
    \begin{itemize}
        \item[(a)] For $n=1$, there is exactly one line $S$ (resp. $s$) of $\cF_{X}$ (resp. $f_{X}$) which tends to $p_j$.
        \item[(b)]For $n=2$. Suppose that $a+ib \coloneqq\lim_{z\to 0}z^2\phi(z)$. Then $a+ib\neq 0$ and there are two cases
        \begin{itemize}
            \item [(b.1)]$b=0$. On one hand, if $a>0$, the lines of $\cF_{X}$ are rays tending to $p_j$ and those of $f_{X}$ circles around $p_j$. On the other hand, if $a<0$, the lines of $\cF_{X}$ are circles around $p_j$ and those of $f_{X}$ are rays tending to $p_j$.
            \item[(b.2)]$b\neq 0$. The lines of $\cF_{X}$ and $f_{X}$ tend to $p_j$.
        \end{itemize}
        \item[(c)] For $n\geq 3$, every line of $\cF_{X}$ and $f_{X}$ tends to $p_j$.
    \end{itemize}
\end{thm}
\begin{proof}
    The satements proofs of Items (a) and (c) are not modified, while case (b) contains a slight improvement with respect to \cite{SotoPrincipalLinesCMC}. We then provide a proof of the case $n=2$. By Proposition \cite[Proposition 4.5]{SotoPrincipalLinesCMC} there exists an isothermic chart $z=re^{i\theta}$ around the end $p_j$ such that
    \begin{equation*}
        \phi(z)=\frac{a+ib}{z^2}=\frac{a+ib}{r^2}e^{-2i\theta}.
    \end{equation*}
    The maximal and minimal principal directions are generated by the vector fields $V_{+},\ V_{-}$ respectively, where
    \begin{gather*}
        V_{\pm}=2\overline{\phi}\pm 2\abs{\phi}=\frac{2(a-ib)}{r^2}e^{2i\theta}\pm \frac{2\sqrt{a^2+b^2}}{r^2}.
    \end{gather*}
    Consider the case $b=0$ and $a>0$. We have
    \begin{gather*}
        V_{+}=\frac{2a}{r^2}\left(\cos 2\theta+1,\sin 2\theta\right)=\frac{4a\cos\theta}{r^2} e^{i\theta},\implies V_{+}||\ e^{i\theta}\implies \cF_X\ \text{rays tending to $p_j$},\\\\  V_{-}=\frac{2a}{r^2}\left(\cos 2\theta-1,\sin 2\theta\right)=\frac{4a\sin\theta}{r^2} e^{i\left(\theta+\frac{\pi}{2}\right)},\implies V_{-}\perp e^{i\theta}\implies f_X\ \text{circles around $p_j$}.
    \end{gather*}
    The case $b=0$ and $a<0$ just swaps the maximal and minimal directions $V_{\pm}$, so the proof is the same. Finally consider the case $b\neq 0$. Consider a curve $z(t)=x(t)+iy(t)$ integrating the maximal principal foliation $\cF_X$, \textit{i.e.} $   V_{+}=\left(\frac{dx}{dt},\frac{dy}{dt}\right)$. Then
    \begin{equation*}
     \frac{dr}{dt}e^{i\theta}+ir\frac{d\theta}{dt}e^{i\theta}=\frac{dx}{dt}+i\frac{dy}{dt}=\frac{2(a-ib)}{r^2}e^{2i\theta}+ \frac{2\sqrt{a^2+b^2}}{r^2}.
    \end{equation*}
    Taking real and imaginary parts we find
    \begin{gather*}
        r^2\frac{dr}{dt}=2\left(a+\sqrt{a^2+b^2}\right)\cos\theta+2b\sin \theta,\quad r^3\frac{d\theta}{dt}=2\left(a-\sqrt{a^2+b^2}\right)\sin\theta-2b\cos \theta,
    \end{gather*}
    which implies that
    \begin{equation}
        \frac{1}{r}\frac{dr}{d\theta}=-\frac{a+\sqrt{a^2+b^2}}{b}\coloneq\gamma\neq 0.
    \end{equation}
    Therefore we see that $r=e^{\gamma \theta}$ so that the lines of $\cF_X$ spiral to the end $p_j$. By a similar argument, when $b\neq 0$ the lines of $f_X$ also spiral to the end $p_j$ and the proof is completed.
\end{proof}
\begin{cor}\label{CriterioclassC}
    Let $X:\Sigma\setminus \{p_1,\ldots, p_N\}\to \bR^3$ be a complete minimal immersion of finite total curvature. $X\in \cC$ if and only if the associated complex function has a pole of order two with non-zero real quadratic limit at each end $p_j$.
\end{cor}
\begin{figure}
     \centering
     \begin{subfigure}[b]{0.25\textwidth}
         \centering
         \includegraphics[width=\textwidth]{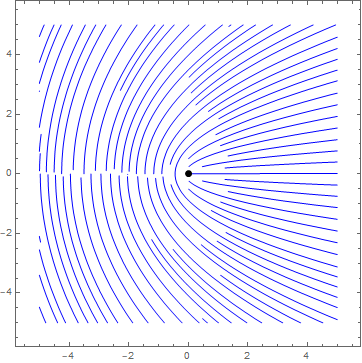}
         \caption{$n=1$.}
         \label{hopf n=-1}
     \end{subfigure}
     \hfill
     \begin{subfigure}[b]{0.25\textwidth}
         \centering
         \includegraphics[width=\textwidth]{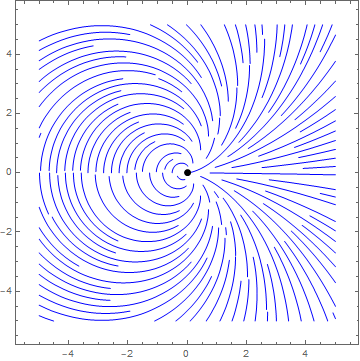}
         \caption{$n=3$.}
         \label{hopf n=-3}
     \end{subfigure}
      \hfill
     \begin{subfigure}[b]{0.25\textwidth}
         \centering
         \includegraphics[width=\textwidth]{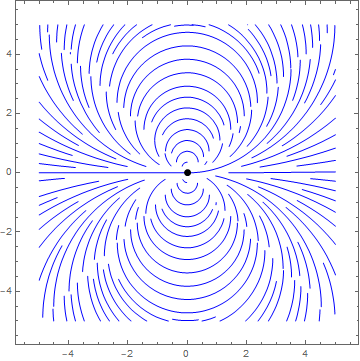}
         \caption{$n=4$.}
         \label{hopf n=-5}
     \end{subfigure}
     \caption{Maximal principal foliation $F_X$ of Theorem \ref{dynamicslinescurvature}, Items (a) and (c).}
     \end{figure}
     \begin{defn}[{\cite{MartinLopezSurvey}}]\label{JorgeMeeksmultiplicities}
    \normalfont{
Let $X:\Sigma\setminus\{p_1,\ldots, p_n\}\to \bR^3$ be a complete minimal immersion with finite total curvature, where $\Sigma$ is a closed Riemann surface. Let $\Phi=(\phi_1,\phi_2,\phi_3)$ be its Weierstrass holomorphic data. The \textit{Jorge-Meeks multiplicity} $\nu_i$ associated with the end $p_i$ is defined by
\begin{equation*}
    \nu_i+1\coloneqq \text{maximum order of a \textit{pole} of $\phi_1,\phi_2,\phi_3$ at $p_i$}.
\end{equation*}
Equivalently,
\begin{equation*}
    \nu_i=\max\{\left(ord_{p_i}(\phi_j)\right)_{-}\}_{j=1}^3-1,
\end{equation*}
where, for any real number $a$, we denote by $a_{-}=\max\{-a,0\}$ the negative part of $a$. 
    }
\end{defn}

We now proceed to make an analysis to relate the multiplicities of Jorge-Meeks formula $\nu_j$ with the order of the Hopf differential at the ends $p_j$. We have the following:
\begin{thm}\label{JorgeMeeksHopf}
 Let $X:\Sigma\setminus \{p_1,\ldots, p_N\}\to \bR^3$ be a complete orientable minimal immersion of finite total curvature. Then at the end $p_j$ we have
    \begin{equation*}
        \nu_j=r_{p_j}(g)-ord_{p_j}Q_H-1.
    \end{equation*}
\end{thm}
\begin{proof}
    Let $\psi:U\subset \bC\to \Sigma$ be a chart centered at $p_j$, and write $n=ord_{p_j}(g)$, $m=ord_{p_j}dh$,
    \begin{equation*}
        g\circ \psi(z)=z^{n}H(z),\quad \psi^{*}dh(z)=z^m\Tilde{H}(z)dz.
    \end{equation*}
    By the completeness condition and equation \eqref{conformal factor immersion}, we must have that
    \begin{equation*}
        m<\abs{n}.
    \end{equation*}
    We have two cases to consider
    \begin{enumerate}
        \item Case $1.$ $\ n\neq0$. By Proposition \ref{dichotomy} $\frac{dg}{g}$ has a simple pole at $p_j$ which implies 
        \begin{equation*}
            ord_{p_j}Q_H=m-1.
        \end{equation*}
        We study the possibilities:
        \begin{itemize}
            \item $0<n,\ 0<m<n$ or  $m<0<n$ implies that
            \begin{equation*}
                \nu_j+1=n-m,\quad r_{p_j}(g)=n-1,
            \end{equation*} so that
            \begin{equation*}
                 \nu_j=n-1-m=n-1-(m-1)-1=r_{p_j}(g)-ord_{p_j}Q_H-1.
            \end{equation*}
            \item $n<0<m<-n$ or $n<0$, $m<0$ implies
            \begin{equation*}
                \nu_j+1=-n-m,\quad r_{p_j}(g)=-n-1,
            \end{equation*}
            so that
             \begin{equation*}
                 \nu_j=-n-1-m=-n-1-(m-1)-1=r_{p_j}(g)-ord_{p_j}Q_H-1.
            \end{equation*}
        \end{itemize}
     \item Case $2.$ $n=0.$ By Proposition \ref{dichotomy} $\frac{dg}{g}$ has a zero of multiplicity $r_{p_j}(g)$ at $p_j$, which implies 
     \begin{equation}\label{order Q}
         ord_{p_j}Q_H=m+r_{p_j}(g).
     \end{equation}
     Since $g(p_j)\neq 0,\infty$ and $m<0$ we have that
     \begin{equation*}
         \nu_j+1=-m,
     \end{equation*}
     which implies by equation \eqref{order Q} that
     \begin{equation*}
         \nu_j=-m-1=r_{p_j}(g)-ord_{p_j}Q_H-1,
     \end{equation*}
     as we wanted to prove.
    \end{enumerate}
\end{proof}
\begin{figure}
     \centering
     \begin{subfigure}[b]{0.29\textwidth}
         \centering
\includegraphics[width=\textwidth]{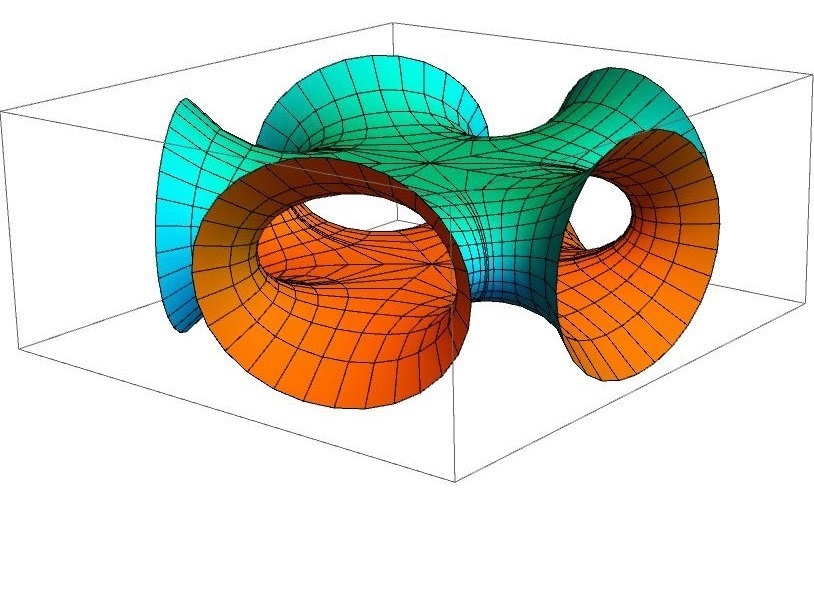}
         \caption{Jorge-Meeks nodoid. Image Credit: M. Weber.}
         \label{Jorge Meeks nodoid}
     \end{subfigure}
     \hfill
     \begin{subfigure}[b]{0.25\textwidth}
         \centering
\includegraphics[width=\textwidth]{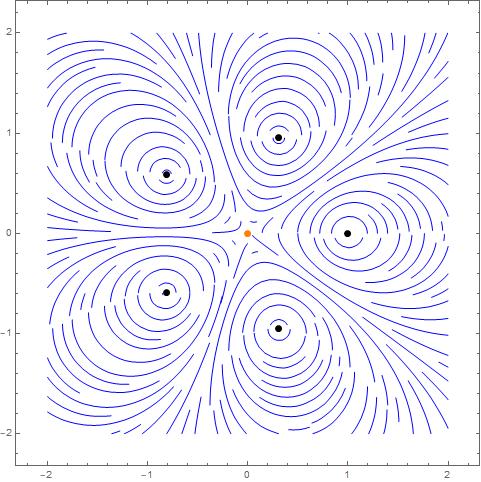}
         \caption{Jorge-Meeks curvature lines for $n=4$ i.e, a sphere with $5$ ends.}
         \label{Jorge Meeks curvature lines}
     \end{subfigure}
      \hfill
     \begin{subfigure}[b]{0.25\textwidth}
         \centering
\includegraphics[width=\textwidth]{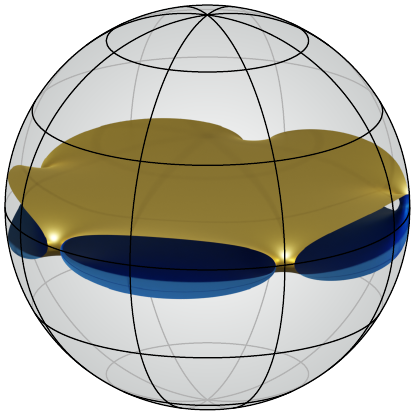}
         \caption{Free boundary minimal nodoids. Image Credit: M. Schulz \cite{SchulzGallery}.}
         \label{free boundary nodoids}
     \end{subfigure}
     \caption{Jorge-Meeks $n+1$-nodoids.}
     \end{figure}

\begin{defn}\label{totalramificationorderends}
    \normalfont{Let $g:\Sigma\setminus\{p_1,\ldots p_N\}\to \overline{\bC}$ be the Gauss map of a finite total curvature minimal surface. We define \textit{the total order of ramification of the Gauss map at the ends $p_j$} as the number
    \begin{equation*}
        r(g)\coloneqq \sum_{j=1}^Nr_{p_j}(g).
    \end{equation*}}
\end{defn}
\begin{rem}\label{RamificationEstimate}
\normalfont{
    Notice that according with our definition, we have $0\leq r(g)\leq N(\deg(g)-1)$.}
\end{rem}
\begin{thm}
    Let $X:M\to \bR^3$ be a non-totally geodesic complete orientable minimal immersion of finite total curvature where $M=\Sigma\setminus\{p_1,\ldots,p_N\}$ and let $Q_H$ be the Hopf differential. Then
    \begin{equation*}
    \cX(\Sigma)=-2\deg(g)-\sum_{j=1}^Nord_{p_j}Q_H+r(g).
    \end{equation*}
\end{thm}
\begin{proof}
     Indeed, as a consequence of Theorem \ref{JorgeMeeksHopf}, we have
\begin{equation*}
    \sum_{j=1}^N(\nu_j+1)=\sum_{j=1}^Nr_{p_j}(g)-\sum_{j=1}^Nord_{p_j}Q_H.
\end{equation*}
The result then follows by an application of the classical Jorge-Meeks formula \cite{MartinLopezSurvey}, 
    \begin{equation*}
    2deg(g)=-\mathcal{X}(\Sigma)+\sum_{i=1}^n(\nu_i+1).
\end{equation*}
\end{proof}

In our particular research of closed lines of curvature, we have the following consequence:
\begin{cor}\label{ClosedCurvatureLineFormula}
  Let $X:\Sigma\setminus\{p_1,\ldots,p_N\}\to \bR^3$ be a non-totally geodesic complete orientable minimal immersion of finite total curvature with closed curvature lines around the ends. Then
    \begin{equation*}
    \cX(\Sigma)=-2\deg(g)+2N+r(g).
    \end{equation*}
\end{cor}
We are ready to prove Theorem \ref{obstructionplanar}:
\begin{thm}\label{PlanarEndsDoNotAdmitClosedLines}
   Let $X:\Sigma\setminus\{p_1,\ldots p_N\}\to \bR^3$ a complete minimal immersion of finite total curvature. If the end $p_j$ is an embedded planar end, then it does not admit a closed curvature line around it.
\end{thm}
\begin{proof}
     Assume by contradiction that the surface has a closed curvature line around $p_j$. After rotating the surface, we can assume that $g$ has a zero at $p_j$. This implies by Proposition \ref{dichotomy} that $\frac{dg}{g}$ has a simple pole at $p_j$ with residue equal to $mult_{p_j}(g)\neq 0$. Since, by Corollary \ref{CriterioclassC}, the Hopf differential has a pole of order two at $p_j$ with $\lim_{z\to 0}z^2\phi(z)\in \bR^{*}$ we must have that $dh$ has a simple pole at $p_j$ with non-zero residue. By Theorem \ref{embeddedends} 
     and Remark \ref{catenoidalvsplanar} this implies that, the embedded end $p_j$ must be catenoidal which is a contradiction.
\end{proof}

 The following Corollary is a classical inequality, which applies not only for the case where the surface contains closed curvature lines.
\begin{cor}\label{Chern-OssermanClosed}
      Let $X:\Sigma\setminus\{p_1,\ldots,p_N\}\to \bR^3$ be a complete orientable minimal immersion of finite total curvature with closed curvature lines around the ends. Denote by $\gamma$ the genus of $\Sigma$. Then we have the inequality
    \begin{equation}\label{Chern-Osserman}
1-\gamma+\deg(g)\geq N.
    \end{equation}
    Moreover the equality holds if and only if each end is catenoidal.
\end{cor}
\begin{proof}
    The inequality \eqref{Chern-Osserman} is the well-known Chern-Osserman inequality, which holds for minimal immersions of finite total curvature without the assumption on the closed curvature lines around the ends. We provide an alternative proof with the extra hypothesis of closed curvature lines. Indeed, the inequality comes from Corollary \ref{ClosedCurvatureLineFormula} and the fact that the total order of ramification of the Gauss map at the ends satisfies $r(g)\geq 0$. Since, $r_{p_j}(g)\geq 0$ for all $j=1,\ldots N$, we see that the equality holds if and only if $r_{p_j}=0$ for all $j$. Since by hypothesis $ord_{p_j}Q_H=-2$, we must have, by Theorem \ref{JorgeMeeksHopf}, that $\nu_j=1$, so the end is embedded by \cite[Proposition 4.3.13]{CarlosThesis} and catenoidal by Theorem \ref{PlanarEndsDoNotAdmitClosedLines}. 
\end{proof}

Using Corollaries \ref{Chern-OssermanClosed}, \ref{ClosedCurvatureLineFormula} and Remark \ref{RamificationEstimate} we can prove an inequality which can be interpreted as saying that we either have a good upper or lower bound on the number of ends of a minimal immersion with closed curvature lines.
\begin{cor}\label{inequalityclosedcurvature}
       Let $X:\Sigma\setminus\{p_1,\ldots,p_N\}\to \bR^3$ be a complete orientable minimal immersion of finite total curvature with closed curvature lines around the ends. Then
       \begin{equation*}
         \frac{2N}{1+\deg(g)}\leq  2-\frac{2\gamma}{1+\deg(g)}\leq N,
       \end{equation*}
       where $\gamma$ is the genus of the closed Riemann surface $\Sigma$.
\end{cor}
\begin{rem}
\normalfont{
    In the particular case when $\deg(g)=1$, we have that $N\leq 2-\gamma\leq N$, which implies $N=2-\gamma$. Since $N\geq 1$, we must have that either $N=1$, $\gamma=1$ or $N=2$, $\gamma=0$. The first case is impossible because $\deg(g)=1$ implies that $g:\Sigma\to \bS^2$ is a diffeomorphism. The second case corresponds to the topology of the catenoid. In fact, since the only complete minimal surfaces of total curvature $-4\pi$ in $\bR^3$ are the Catenoid and Enneper's surface \cite{OssermanSurvey}, we see that the only surface with $\deg(g)=1$ and closed lines of curvature is the Catenoid.}
\end{rem}

\section{Non-orientable minimal surfaces with finite total curvature}\label{Section 4.4}

The study of non-orientable minimal surfaces started in 1876 with the discovery by L. Henneberg \cite{Henneberg} of a minimal Möbius band in $\bR^3$ with two branch points and total curvature $-2\pi$. We begin the discussion with a survey of the constructions and obstructions found in the literature:
\begin{itemize}
\item W. Meeks constructed the first complete non-orientable minimal immersion of a Möbius band of finite total curvature $-6\pi$ in the Euclidean space $\bR^3$ free of branch points \cite{Meeks}.
\item  M. Oliveira found a generalization of Meeks Möbius band with higher total curvature \cite{Oliveira}.
\item A. Barros found three one parameter families of Möbius bands with total curvature $-10\pi$ depending on the multiplicity of the Gauss map at the single end \cite{Barros}.
\item T. Ishihara described the moduli space of all the minimal Möbius bands \cite{Toru}.

\item M. Oliveira constructed two minimal projective spaces with two ends \cite{Oliveira}.
\item A. Barros generalized Oliveira's two-ended surface to higher total curvature \cite{Barros}.
\item S. Zhang constructed all the twice-punctured minimal projective spaces with total curvature $-10\pi$, $max\{mult_a(g),mult_b(g)\}=2$ and parallel ends  \cite{Zhang}.
\item E. Toubiana found a one parameter family of complete minimal twice-punctured projective spaces of total curvature $-10\pi$ interpolating between the Oliveira's two-ended Möbius bands \cite{Toubiana}.

\item M. Oliveira constructed a minimal three-punctured projective space with total curvature $-14\pi$  \cite{Oliveira}.
\item A. Barros generalized Oliveira's three-ended Möbius band to higher total curvature \cite{Barros}. 
\item R. Bryant found a minimal immersion of a three-punctured projective space with total curvature $-10\pi$ \cite{Bryant}.
\item R. Kusner discovered complete minimal projective spaces with an odd number of ends and finite total curvature \cite{Kusner}.
\item T. Ishihara proved that there does not exist a complete minimal immersion of finite total curvature in $\bR^3$ with the topology of a projective space and three embedded parallel ends \cite{Torunonexistence}.

\item S. Kato and K. Hamada constructed a $1-$parameter family of minimal projective spaces with $(n+1)$ ends of catenoidal type for any odd integer $n\geq 3$, and with total curvature corresponding to equality in Chen-Osserman's inequality \cite{Kato}. 
\item S. Kato and K. Hamada \cite{Kato} proved a non-existence theorem of minimal $n+1$-punctured projective spaces with catenoidal ends realizing the equality in the Chern-Osserman inequality and with $\bZ^n$ symmetry, for any even integer $n\geq 2$ \cite{Kato}.

\item  F. López discovered a minimal Klein bottle with a single end and total curvature $-8\pi$ \cite{Lopez}.
\item T. Ishihara has also constructed complete minimal Klein bottles in $\bR^3$ with higher total curvature \cite{ToruKleinBottles}.
\item F. Martin generalized the Lopez Klein bottle by constructing complete minimal non-orientable minimal surfaces with finite total curvature a single end and any genus \cite{MartinSurfaces}.

\item Further constructions were done by R. Kusner and N. Schmitt in the case of higher topological complexity \cite{KusnerSchmitt}.
\end{itemize}

Consider the antiholomorphic map $\cJ: \bS^2\to \bS^2$ and the canonical projection $\pi_0$
\begin{equation*}
\cJ(z)=-\frac{1}{\overline{z}},\quad \pi_0:\bS^2\to \frac{\bS^2}{\langle \cJ\rangle}\equiv \mathbb{RP}^2.
\end{equation*} 

The global version of the Weierstrass representation for non-orientable minimal immersions is the following:
 \begin{thm}[{Non-orientable Weierstrass Representation Theorem}\cite{Meeks},\cite{MartinThesis}]\label{globalWeierstrassNonOriented}
     Let $M$ be a Riemann surface and let $(g,\eta)$ be a Weierstrass data over $M$ satisfying the hypothesis of Theorem \ref{globalWeierstrassOriented}. Suppose that there exists an antiholomorphic involution $\tau:M\to M$ without fixed points satisfying
     \begin{equation}\label{transformationalpropertiesWeierstrass}
         g\circ \tau=\cJ\circ g,\quad \tau^{*}\eta=-\overline{g^2\eta}.
     \end{equation}
     Then the minimal immersion $X:M\to \bR^3$ of Theorem \ref{globalWeierstrassOriented} induces a non-orientable minimal immersion $X':M'\to \bR^3$ such that 
     \begin{equation*}
         X'\circ \pi=X,
     \end{equation*}
     where
     \begin{equation*}
         M'=M/\langle\tau\rangle,\quad \pi:M\to M' .
     \end{equation*}
 \end{thm}
 
Even though for non-orientable surfaces there does not exist a well defined continuous normal vector field on the whole surface, there is a well defined Gauss map $g:M\to \overline{\bC}$ defined on the orientable double cover $M$ of $M'$. Moreover by equation \eqref{transformationalpropertiesWeierstrass} there exists a well defined map $\mathfrak{g}:M'\to \mathbb{RP}^2$ which makes commutative the following diagram:
\begin{center}
\begin{tikzcd}

M \arrow[rr, "g"] \arrow[dd, "\pi"] &                                               & \overline{\mathbb{C}} \arrow[dd, "\pi_0"] \\
                                    & {} \arrow[loop, distance=2em, in=55, out=125] &                                           \\
M' \arrow[rr, "\mathfrak{g}"]                       &                                               & \mathbb{RP}^2                            
\end{tikzcd}
\end{center}
We will call $\mathfrak{g}$ the Gauss map of the non-orientable minimal surface $M'.$  Since $\deg(\pi)=\deg(\pi_0)=2$, it holds that $\deg(g)=\deg(\mathfrak{g})$.

F. Martin, in his doctoral thesis \cite{MartinThesis}, generalized the Jorge-Meeks formula for non-orientable minimal surfaces. We can also provide a generalization of Corollary \ref{ClosedCurvatureLineFormula} for the context of non-orientable minimal immersions.

\begin{prop}\label{BranchPointsComeInPairs}
   Suppose that $g:\Sigma\to \bS^2$ is the extended Gauss map of the double cover of a non-orientable minimal surface of finite total curvature with $\tau:\Sigma\to \Sigma$ the antiholomorphic involution. Then 
    \begin{equation*}
       \forall p\in \Sigma,\quad  r_p(g)=r_{\tau(p)}(g).
    \end{equation*}
\end{prop}
\begin{proof}
    Consider a complex chart $\phi:D_{\epsilon}\to \Sigma$ centered at $p\in \Sigma$. Then by Lemma \ref{Chartrelatedpoint} we have the complex chart $\hat{\phi}= \tau\circ \phi \circ T: D_{\epsilon}\to \Sigma$ centered at $\tau(p)$ is compatible with the Riemann surface structure. By Lemma \ref{function in terms of order} there exist holomorphic functions $H,\Tilde{H}$ with $H(0),\Tilde{H}(0)\neq 0$ such that
    \begin{equation*}
        g(\phi(z))=z^{ord_p(g)}H(z),\quad g(\hat{\phi}(z))=z^{ord_{\tau(p)}(g)}\Tilde{H}(z).
    \end{equation*}
    Since $g\circ \tau=\cJ\circ g$ we must have that
    \begin{gather*}
        z^{ord_{\tau(p)}(g)}\Tilde{H}(z)=g(\hat{\phi}(z))=-\frac{1}{\overline{g(\phi(\overline{z})}}=-\frac{1}{z^{ord_p(g)}\overline{H(\overline{z})}}
    \end{gather*}
    so that
    \begin{equation*}
     z^{ord_p(g)}   z^{ord_{\tau(p)}(g)}=-\frac{1}{\overline{H(\overline{z})}\Tilde{H}(z)}.
    \end{equation*}
    Since the function on the right side does not vanish at $z=0$ we must have that $ord_p(g)=-ord_{\tau(p)}(g)$ so that
    \begin{equation*}
        r_p(g)=mult_p(g)-1=\abs{ord_p(g)}-1=\abs{ord_{\tau(p)}(g)}-1=mult_{\tau(p)}(g)-1= r_{\tau(p)}(g),
    \end{equation*}
    as we wanted to prove.
\end{proof}

We introduce the following definition:
\begin{defn}
\normalfont{
    Let $X':\Sigma'\setminus\{p_1,\ldots, p_N\}$ be a complete non-orientable minimal immersion of finite total curvature. Let $\Sigma$ be the orientable double cover of $\Sigma'$ with antiholomorphic involution $\tau:\Sigma\to \Sigma$ and associated orientable minimal immersion $X:\Sigma\setminus\{p_1,\ldots,p_N,\tau(p_1),\ldots, \tau(p_N)\}\to \bR^3$ with extended Gauss map $g:\Sigma\to \bS^2$. We define the total ramification order at the ends of the Gauss map $\mathfrak{g}:\Sigma\to \mathbb{RP}^2$ of $X'$ as
    \begin{equation*}
        r(\mathfrak{g})\coloneqq \frac{1}{2}r(g).
    \end{equation*}
    }
\end{defn}
\begin{rem}
\normalfont{
The total ramification order at the ends of $\mathfrak{g}$ is an integer. In fact by Definition \ref{totalramificationorderends} and Proposition \ref{BranchPointsComeInPairs} we have
    \begin{equation*}
       r(\mathfrak{g})\coloneqq \frac{1}{2}r(g)=\frac{1}{2}\left(\sum_{j=1}^Nr_{p_j}(g)+\sum_{j=1}^Nr_{\tau(p_j)}(g)\right)= \sum_{j=1}^Nr_{p_j}(g).
    \end{equation*}}
\end{rem}
\begin{rem}
\normalfont{
    Notice that according with Remark \ref{RamificationEstimate} applied to the case under consideration, we have $0\leq r(g)\leq 2N(\deg(g)-1)$ which implies $r(\mathfrak{g})\leq N(\deg(\mathfrak{g})-1)$.}
\end{rem}
\begin{cor}[{Non-orientable version of Corollary \ref{ClosedCurvatureLineFormula}}]\label{closedlineformulanonoriented}
 Let $X':\Sigma'\setminus\{p_1,\ldots,p_N\}\to \bR^3$ be a complete non-orientable minimal immersion of finite total curvature with closed curvature lines around the ends $p_j$. Then
    \begin{equation*}
    \cX(\Sigma')=-\deg(\mathfrak{g})+2N+ r(\mathfrak{g}).
    \end{equation*}
\end{cor}
\begin{proof}
    In fact consider the double surface $X:\Sigma\setminus\{q_1,\ldots,q_N,\tau(q_1),\ldots \tau(q_N)\}\to \bR^3$ with $2N$ ends and closed curvature lines. By Corollary \ref{ClosedCurvatureLineFormula} we have
    \begin{equation*}
        \cX(\Sigma)=-2\deg(g)+2(2N)+r(g),
    \end{equation*}
    so that
    \begin{equation*}
        2\cX(\Sigma')=-2\deg(\mathfrak{g})+4N+2r(\mathfrak{g}).
    \end{equation*}
    and the proof is completed.
\end{proof}

\begin{cor}[{Non-orientable version of Corollary \ref{Chern-OssermanClosed}}]
      Let $X':\Sigma'\setminus\{p_1,\ldots,p_N\}\to \bR^3$ be a complete non-orientable minimal immersion of finite total curvature with closed curvature lines around the ends $p_j$. Denote by $\gamma$ the genus of the oriented double cover of $\Sigma'$. Then we have the inequality
    \begin{equation*}
1-\gamma+\deg(\mathfrak{g})\geq 2N.
    \end{equation*}
    Moreover the equality holds if and only if each end is catenoidal.
\end{cor}

We now provide a generalization of the transformational property \cite[Theorem 4.5]{Carlos} of the Hopf differential for arbitrary genus and an application to the dynamics of the lines of curvature associated to the double surface of a non-orientable minimal immersion.
\begin{prop}[{Transformational law of the Hopf differential}]\label{transformationHopfgeneralized}
    Let $X':M'\to \bR^3$ be a complete non-orientable minimal surface with finite total curvature and $X:M\to \bR^3$ the associated minimal immersion on the orientable double cover. Let $\tau:M\to M$ be the antiholomorphic involution such that $M'=M/\langle \tau \rangle$. Then the Hopf differential $Q_{H}$ of $X$ satisfies the transformational law
    \begin{equation*}
        \tau^{*}Q_{H}=-\overline{Q}_H.
    \end{equation*}
\end{prop}
\begin{proof}
    In fact, let $(g,\eta)$ be the Weierstrass data defined on $M$. Recall that
    \begin{equation*}
        Q_H=-\frac{1}{2}dg\otimes\eta.
    \end{equation*}
    Then according with the transformational properties of \eqref{transformationalpropertiesWeierstrass}, we calculate
    \begin{gather*}
    \begin{split}
        \tau^{*}Q_H&=\tau^{*}\left(-\frac{1}{2}dg\otimes \eta\right)=-\frac{1}{2}d(g\circ \tau)\otimes (\tau^{*}\eta)=-\frac{1}{2}d\left(-\frac{1}{\overline{g}}\right)\otimes \left(-\overline{g^2\eta}\right)\\&=
         \tau^{*}Q_H=-\frac{1}{2}\frac{d\overline{g}}{\overline{g}^2}\otimes \left(-\overline{g^2\eta}\right)=\frac{1}{2}d\overline{g}\otimes \overline{\eta}=-\overline{Q}_H.
         \end{split}
    \end{gather*}
    as we wanted to prove.
\end{proof}
\begin{cor}\label{LinesCurvaturerelatedpoints}
     Let $X':\Sigma'\setminus\{p_1.\ldots, p_N\}\to \bR^3$ be a complete non-orientable minimal surface with finite total curvature, and let $X:\Sigma\setminus \{p_1,\ldots p_N,\tau(p_1),\ldots \tau(p_N)\}\to \bR^3$ be the associated minimal immersion on the orientable double cover, where $\tau:\Sigma\to \Sigma$ is the antiholomorphic involution such that $\Sigma'=\Sigma/{\langle \tau \rangle}$. Suppose that the maximal principal foliation $\cF_{X}$ (respectively the minimal principal foliation $f_X$) has a closed curvature line around the end $p_j$. Then $\cF_X$ (respectively $f_X$) has a radial curvature line around the end $\tau(p_j).$
\end{cor}
\begin{proof}
    Consider an isothermal atlas $(U_{\alpha},\psi_{\alpha})$ of $\Sigma$, where $\psi_{\alpha}:U_\alpha\subset \bC\to \Sigma$. Then the Hopf differential $Q_H$ satisfies by definition
    \begin{equation*}
        \psi_\alpha^{*}Q_H=\phi_\alpha(z) dz \otimes dz,\quad \phi_{\alpha}(z)\coloneq II(\p_z\psi_\alpha,\p_z\psi_\alpha).
    \end{equation*}
    Fix an end $p_j$ and let $\psi_\alpha$ be a chart centered at that point. Then according with Lemma \ref{Chartrelatedpoint} a compatible complex chart around $\tau(p_j)$ is $\hat{\psi}_\alpha\coloneqq \tau\circ \psi_\alpha \circ T$ where $T:U_\alpha\to U_\alpha$ is the conjugation map. Let us denote by $(U_\beta,\psi_\beta)$ this complex chart where $U_\beta=U_\alpha$ and $\psi_\beta=\hat{\psi}_\alpha$. Suppose by hypothesis that the maximal principal foliation $\cF_X$ has a closed curvature line around $p_j$. Therefore by Corollary \ref{CriterioclassC} this implies that $\phi_\alpha(z)$ has a pole of order two at $0$ with negative quadratic limit
    \begin{equation}\label{complexassociatefunction}
        \phi_\alpha(z)=\frac{a_{\alpha}}{z^2}+\frac{b_\alpha}{z}+\text{holomorphic function},\quad  a_{\alpha}\in \bR^{-}.
    \end{equation}
    On the other hand we have
    \begin{equation*}
\psi_\beta^{*}Q_H=\phi_\beta(z)dz\otimes dz.
    \end{equation*}
By Proposition \ref{transformationHopfgeneralized} we find that
\begin{gather*}
\psi_\beta^{*}Q_H=T^{*}\left(\psi_\alpha^{*}\left(\tau^{*}Q_H\right)\right)=T^{*}\left(\overline{\psi_\alpha^{*}\left(\overline{\tau^{*}Q_H}\right)}\right)=-T^{*}\left(\overline{\psi_\alpha^{*}Q_H}\right),
\end{gather*}
which implies that
\begin{equation*}
    \phi_{\beta}(z)dz\otimes dz=-T^{*}\left(\overline{\phi_\alpha(z)dz\otimes dz}\right)=-\left(\overline{\phi_{\alpha}(z)}\circ T\right) T^{*}\overline{dz}\otimes T^{*}\overline{dz}=-\overline{\phi_\alpha(\overline{z})}dz\otimes dz.
\end{equation*}
Therefore
\begin{equation*}
    \phi_\beta(z)=-\overline{\phi_\alpha(\overline{z})}.
\end{equation*}
which implies by equation \eqref{complexassociatefunction} that $\phi_\beta(z)$ has a pole of order two at $z=0$ and moreover 
\begin{equation*}
    a_\beta=-\overline{a}_{\alpha}\in \bR^{+}.
\end{equation*}
As a consequence, by Theorem \ref{dynamicslinescurvature}, the maximal principal foliation $\cF_X$ has a radial behavior at $\tau(p_j)$.
\end{proof}

\subsection{Single-ended minimal surfaces with a closed curvature line}

While looking for a model of a non-orientable free boundary minimal immersion in the unit ball $\bB^3$, we investigated the question of whether or not there exists a complete non-orientable minimal surface of finite total curvature in $\bR^3$ with a closed curvature line around its end. We next prove Theorem \ref{nonexistenceoneendanygenus}:
\begin{thm}\label{Nonexistencetheorem}
    There is no complete single-ended minimal surface with finite total curvature in $\bR^3$ which contains a closed curvature line around its end.
\end{thm}

\begin{proof}
We provide the proof in the non-orientable case, because the orientable is simpler and follows the same strategy. Suppose, by contradiction, that there exists a complete non-orientable minimal immersion,
     \begin{equation*}
         X':\Sigma'\setminus \{P\}\to \bR^3,
     \end{equation*}
     of finite total curvature and a closed curvature line around $P$. Consider the oriented double cover $\Sigma$ of $\Sigma'$ with canonical projection map $\pi:\Sigma'\to \Sigma$ and antiholomorphic involution map $\tau:\Sigma\to \Sigma$. We then have by composition an orientable minimal immersion 
     \begin{equation*}
         X\coloneqq X' \circ \pi:\Sigma\setminus \{Q_1,Q_2\}\to \bR^3,
     \end{equation*}
     with finite total curvature and closed curvature lines around both ends. After rotating the surface, we can assume by equation \eqref{transformationalpropertiesWeierstrass} that the Gauss map of $X$ has a zero at the end $Q_1 \in \Sigma$ and a pole at the other end $Q_2=\tau(Q_1) \in  \Sigma$ with the same multiplicity. Then the Hopf differential can be written as \begin{equation*}
    Q_H=-\frac{1}{2}\frac{dg}{g}\otimes dh,\quad dh=g\eta,
\end{equation*}
where $(g,\eta)$ is the associated Weierstrass data of the minimal immersion $X$. Notice, that in a local chart $(U_j,\psi_j)$ around each puncture $Q_j,\ j=1,2$, there exist integers $N_1=-N_2$ and holomorphic functions $H_j$ with $H_j(0)\neq 0$ so that
\begin{equation*}
    g\circ \psi_j(z)=z^{N_j}H_j(z),
\end{equation*}
and in any case we see that the one-form $\frac{dg}{g}$ has a pole of order one at each end $Q_1,Q_2$,
\begin{gather*}
    \frac{d(g\circ \psi_j)}{(g\circ \psi_j)}(z)=\left(\frac{N_j}{z}+\Tilde{H}_j(z)\right)dz,
\end{gather*}
with $\Tilde{H}_1,\Tilde{H}_2$ holomorphic. Since we have closed curvature lines around each end, we must have, by Corollary \ref{CriterioclassC}, that the Hopf differential has a pole of order exactly two in a local chart $z=0$ at each end,  with
\begin{equation*}
    \lim_{z\to 0}z^2\phi(z)\in \bR^*.
\end{equation*}
Therefore $dh$ must have a simple pole with non-zero real residues $a_1,a_2\in \bR^{*}$ at both $Q_1,Q_2$ respectively. We next study the transformational property of the height differential through the antiholomorphic involution $\tau$ on $\Sigma$:
\begin{gather}\label{transformationheightdiferential}
    \tau^{*}dh=\left(g\circ \tau\right)\left(\tau^{*}\eta\right)=\left(-\frac{1}{\overline{g}}\right)\left(-\overline{g^2\eta}\right)=\overline{g\eta}=\overline{dh}.
\end{gather}
By Lemma \ref{Residuerelatedpoints} we have that
\begin{equation*}
\eval{Res}_{Q_1}\overline{\tau^{*}dh}=\overline{\eval{Res}_{Q_2}dh}
\end{equation*}
so using equation \eqref{transformationheightdiferential} we obtain
\begin{equation*}
\eval{Res}_{Q_1}dh=\overline{\eval{Res}_{Q_2}dh}.
\end{equation*}
This implies that
\begin{gather*}
    a_1=\eval{Res}_{Q_1}dh=\overline{\eval{Res}_{Q_2}dh}=\overline{a_2}=a_2.
\end{gather*}
Using the residue Theorem \cite{RiemannSurfacesGirondo} we obtain
\begin{equation*}
    0=\sum_{Q\in \Sigma} Res_{Q}dh=a_1+a_2=2a_1
\end{equation*}
implying that $a_1=a_2=0$, which is a contradiction.     
\end{proof}

Therefore, the easiest topology to look for a non-orientable minimal surface with closed curvature lines around the ends is a projective space punctured at two or more points. In the last three sections, we describe the partial results we obtained when analyzing those surfaces.
\subsection{Generalities of non-orientable minimal punctured projective spaces}

Oliveira established the general formula for the Gauss map of a non-orientable minimal $N$-punctured projective space. Later on Barros and Zhang, working independently, established the formula for the one-form $\eta$ of the Weierstrass data associated to a twice-punctured projective space. The purpose of the next theorem is to state an equivalent formulation, and generalize their formulas to any number of punctures.
\begin{thm}[{Oliveira\cite{Oliveira}, Barros \cite{Barros}, Zhang \cite{Zhang}}]\label{Puncturedprojectivespaces}
Let $X':\mathbb{RP}^2\setminus\{p_1,\ldots, p_N\}\to \bR^3$ be a complete minimal immersion of finite total curvature $-2\pi m$. Let $(g,\eta)$ be the Weierstrass data associated to the orientable double cover surface. Then up to equivalence the pair $(g,\eta)$ is of the form
     \begin{equation*}
        g(z)=z^{k_0}\prod_{j=1}^{N-1}\frac{(\overline{a_j}z+1)^{k_j}}{(z-a_j)^{k_j}}\prod_{j=1}^{l}\frac{(\overline{b_j}z+1)}{(z-b_j)},
    \end{equation*}
    \begin{equation}\label{oneformpunctureprojectivespace}
        \eta=\frac{i}{z^{n_0}}\prod_{j=1}^{N-1}\frac{(z-a_j)^{2k_j}}{(z-a_j)^{n_j}(\overline{a_j}z+1)^{n_j}}\prod_{j=1}^{l}(z-b_j)^2dz,
    \end{equation}
    where
 \begin{equation*}
     m\geq3\ \text{odd},\ k_0, k_1,\ldots k_{N-1}\geq 0,\ n_0,\ldots n_{N-1}\geq 2,\ l\geq 0
 \end{equation*}
 are integers satisfying
 \begin{equation*}
     m+1=n_0+n_1+\ldots n_{N-1},\quad l=m-k_0-k_1-\ldots -k_{N-1}
 \end{equation*}
 and
 \begin{equation*}
   a_0=0,\ a_1\ldots a_{N-1}\in \bC,\ a_j\neq a_k,\  a_j\overline{a_k}\neq -1\quad b_j\in \bC,\  b_j\neq 0,\infty,a_k,-\frac{1}{\overline{a_k}},\quad b_j\overline{b_k}\neq -1.
 \end{equation*}
 Conversely, if the Weierstrass data $(g,\eta)$ defined in this way is free of real periods, then through the non-orientable Weierstrass representation it constructs a complete minimal immersion $$X':\mathbb{RP}^2\setminus\{[0],\ldots [a_{N-1}]\}\to \bR^3$$ of finite total curvature $-2\pi m$.
\end{thm}

From the proof of the previous result, see \cite[Proof of Theorem 4.4.25]{CarlosThesis}, we are able to state the following:
\begin{cor}\label{jorgemeeksmultiplicitiespunctureprojectivespaces}
    The Jorge-Meeks multiplicities of the complete minimal immersion  $X':\mathbb{RP}^2\setminus\{p_1,\ldots, p_N\}\to \bR^3$ of Theorem \ref{Puncturedprojectivespaces} are
    \begin{equation*}
        \nu_j=n_j-1.
    \end{equation*}
\end{cor}

Combining this calculation with our previous analysis of the ramification of the Gauss map at the ends, we are able to prove the following restriction
\begin{prop}\label{ramificationorderpunctureprojectiveclosedlines}
     Let $X':\mathbb{RP}^2\setminus\{p_1,\ldots, p_N\}\to \bR^3$ be a complete minimal immersion of finite total curvature of Theorem \ref{Puncturedprojectivespaces} with closed curvature lines. Then
    \begin{equation*}
        r_{p_j}(g)=n_j-2.
    \end{equation*}
    Moreover, if $k_j\geq 1$, then $k_j=n_j-1.$
\end{prop}
\begin{proof}
    In fact, by Theorem \ref{JorgeMeeksHopf}, we know that
    \begin{equation*}
        \nu_{p_j}=r_{p_j}(g)-ord_{p_j}Q_H-1=r_{p_j}+1.
    \end{equation*}
    Using Corollary \ref{jorgemeeksmultiplicitiespunctureprojectivespaces} we obtain that $r_{p_j}(g)=n_j-2$. In the special case when $k_j\geq 1$, we have that $g(p_j)$ has either a zero or a pole at the end $p_j$ (depending on whether $p_j=a_j$ or $p_j=-\frac{1}{\overline{a_j}})$ and the Gauss map can be written in a chart centered at $p_j$ as $g(z)=z^{\pm k_j}H(z)$ with $H$ holomorphic $H(0)\neq 0$. Hence, in any case $mult_{p_j}(g)=\abs{ord_{p_j}(g)}=k_j$, and therefore $r_{p_j}=k_j-1$, which implies, by the first part of the proof, that $k_j=n_j-1.$
\end{proof}

We make the following definition:
 \begin{defn}\label{alphabetagamma}
 \normalfont{
     Let $(g,\eta)$ be the Weierstrass data of a orientable minimal immersion of finite total curvature $X:\Sigma\setminus\{p_1,\ldots,p_N\}\to \bR^3$ and denote by $dh=g\eta$ the height differential. Then, for each end $p_j$, we define
     \begin{equation*}
    \alpha(p_j)\coloneq \eval{Res}_{p_j}\frac{dh}{g},\quad \beta(p_j)\coloneq \eval{Res}_{p_j}gdh,\quad  \gamma(p_j)\coloneqq Res_{p_j}dh.
\end{equation*}
}
 \end{defn}

In the non-orientable setting, we have that those residues have certain transformation properties:
 \begin{lema}\label{alphabetagammarelatedpoints}
     Let $\Sigma$ be a closed Riemann surface with an antiholomorphic involution $\tau:\Sigma\to \Sigma$. Suppose further that $(g,\eta)$ are Weierstrass data defined on $\Sigma\setminus\{p_1,\ldots p_N\}$ satisfying only the transformational property \eqref{transformationalpropertiesWeierstrass}. Then for all $j=1,\ldots, N$
\begin{equation*}
    \alpha(\tau(p_j))+\overline{\beta(p_j)}=0,\quad \gamma(\tau(p_j))=\overline{\gamma(p_j)}.
\end{equation*}
 \end{lema}

When $\Sigma$ has the topology of a sphere, there is a simplification of the period problem 
 \begin{lema}\label{equivalentperiodcondition}
     Suppose that $\Phi$ are Weierstrass data defined on $M=\bS^2\setminus\{p_1,\ldots,p_N\}$. Then 
     \begin{equation*}
           Re\ \int_\gamma\Phi=0,\ \forall\gamma\in H_1(M,\bZ)\iff   \gamma(p_j)\in \bR,\quad \alpha(p_j)+\overline{\beta(p_j)}=0,\quad \forall j=1,\ldots ,N.
     \end{equation*}
 \end{lema}

Next we adapt a lemma originally due to S. Zhang for twice punctured projective spaces in $\bR^3$, which provides a simplification tool to solve the period problem.
\begin{lema}[{S. Zhang \cite{Zhang}}]\label{simplifiedresidueproblem}
    Let $\Phi$ be Weierstrass data on $M=\bS^2\setminus\left\lbrace 0,\infty,a_1,\ldots,a_{N-1},-\frac{1}{\overline{a_1}},\ldots, -\frac{1}{\overline{a_{N-1}}}\right\rbrace$, satisfying the immersion condition \eqref{immersion condition} and the transformational property $\tau^{*}\Phi=\overline{\Phi}$ where $\tau=\cJ$ is the antiholomorphic involution in $M$. Then the period problem is solved if and only if
    \begin{gather*}
           \forall 1\leq j\leq N-1,\quad \gamma(0),\gamma(a_j)\in \bR,\quad \alpha(0)+\overline{\beta(0)}=0,\quad \alpha(a_j)+\overline{\beta(a_j)}=0.
    \end{gather*}
    In that case there exists a minimal immersion $X':\mathbb{RP}^2\setminus\{p_1,\ldots p_N\}\to \bR^3$ through the non-orientable Weierstrass representation Theorem \ref{globalWeierstrassNonOriented} with Weierstrass data $\Phi$.
\end{lema}
\subsection{Non-orientable minimal projective spaces with two ends}

We now focus on complete minimal immersions $X':M'\to \bR^3$ of the twice punctured projective space $M'=\mathbb{RP}^2\setminus\{p_1,p_2\}$ into Euclidean space. After a conformal diffeomorphism, we can always assume, without loss of generality, that the ends of the orientable double cover are positioned at the points $0,\infty,a,-\frac{1}{a}$ with $a\in \bR^{*}$. Therefore the oriented double minimal immersion associated to $X'$ can be written as 
\begin{equation*}
    X:M\to \bR^3,\quad M=\bS^2\setminus\left\lbrace0,\infty,a,-\frac{1}{a}\right\rbrace.
\end{equation*}
By Theorem \ref{Puncturedprojectivespaces} the Weierstrass data of such two ended non-orientable surface with total curvature $-2\pi m$ has to be given by 
\begin{gather}\label{Weierstrasstwopunctures}
    g=\frac{z^{k_0}(az+1)^{k_1}}{(z-a)^{k_1}}\prod_{j=1}^{m-k_0-k_1}\frac{(\overline{b_j}z+1)}{(z-b_j)}, \quad \eta=i\frac{(z-a)^{2k_1}}{z^{n_0}(z-a)^{n_1}(az+1)^{n_1}}\prod_{j=1}^{m-k_0-k_1}(z-b_j)^2dz,
    \end{gather}
where $k_0,k_1\geq0,\ n_0,n_1\geq 2$ are integers with $k_0+k_1\leq m$, $n_0+n_1=m+1$ and $b_j\neq 0, \infty, a, -\frac{1}{a}.$

We cannot aim at finding such a surface with catenoidal ends by the non-existence theorem of R. Kusner of complete non-orientable minimal surfaces of finite total curvature and two embedded ends \cite{Kusner}. We also prove the obstruction Theorem \ref{theorem6}:
\begin{thm}\label{Nonexistencetwopuncturesclosed}
   There does not exist a complete minimal immersion of twice-punctured projective space with parallel ends and finite total curvature with closed curvature lines around its ends. 
\end{thm}
\begin{proof}
    Assume by contradiction that such an immersion exists. After a rotation of the surface we can assume, without loss of generality, that $k_0,k_1\geq 1$. By Proposition \ref{ramificationorderpunctureprojectiveclosedlines} $n_j=k_j+1$ which implies that by Theorem \ref{Puncturedprojectivespaces}
    \begin{equation*}
        m-k_0-k_1=(n_0+n_1-1)-(n_0-1)-(n_1-1)=1,
    \end{equation*}
    and therefore there exists $b\neq 0,\infty,a,-\frac{1}{a}$ such that the Weierstrass data becomes
    \begin{equation*}
            g=\frac{z^{n_0-1}(az+1)^{n_1-1}}{(z-a)^{n_1-1}}\frac{(\overline{b}z+1)}{(z-b)}, \quad \eta=i\frac{(z-a)^{n_1-2}}{z^{n_0}(az+1)^{n_1}}(z-b)^2dz.
    \end{equation*}
    Notice that if $p_j$ denotes any of the ends, then by Theorem \ref{JorgeMeeksHopf}
    \begin{equation*}
        ord_{p_j}Q_H=r_{p_j}-1-\nu_j=(n_j-2)-1-(n_j-1)=-2,
    \end{equation*}
    so the Hopf differential has indeed a pole of order two at each end. Since the ends are parallel we have that $\frac{dg}{g}$ has a simple pole with real residue at each end, indeed
    \begin{equation*}
        \frac{dg}{g}=\frac{n_0-1}{z}+a\frac{n_1-1}{az+1}+\frac{\overline{b}}{\overline{b}z+1}-\frac{n_1-1}{z-a}-\frac{1}{z-b}.
    \end{equation*}
    On the other hand, the height differential also has a simple pole at each end
    \begin{equation*}
        dh=i\frac{(z-b)(\overline{b}z+1)}{z(z-a)(az+1)}.
    \end{equation*}
    This implies, in the particular case of parallel ends, that the residue conditions of the immersion for $\gamma$ is equivalent to the condition on the Hopf differential to have closed curvature lines around the ends, more precisely
    \begin{equation*}
        \gamma(0),\gamma(a)\in \bR\iff \lim_{z\to 0}z^2Q_H,\ \lim_{z\to a}(z-a)^2Q_H\in \bR\quad \textit{(for parallel ends)}.
    \end{equation*}
  Since by assumption both conditions are satisfied, and since $a\in \bR$ these two conditions imply that
  \begin{equation*}
      b\in i\bR^{*},\quad (a-b)(\overline{b}a+1)\in i\bR^{*}.
  \end{equation*}
  Writing explicitly the real and imaginary parts, we check that the only solutions to these equations are $b=\pm i$, \textit{i.e.} $b$ is purely imaginary with $\overline{b}=-b$ in any case. Notice that
  \begin{equation*}
      gdh=i\frac{z^{n_0-2}(az+1)^{n_1-2}}{(z-a)^{n_1}}(-bz+1)^2,\quad \frac{dh}{g}=i\frac{(z-a)^{n_1-2}(z-b)^2}{z^{n_0}(az+1)^{n_1}}.
  \end{equation*}
So we see that
\begin{equation*}
    \alpha(a)=0,\quad \beta(0)=0.
\end{equation*}
Therefore the residue conditions from Lemma \ref{simplifiedresidueproblem} imply that
\begin{equation*}
    \alpha(0)=0,\quad \beta(a)=0.
\end{equation*}
Since $gdh$ has a pole at $z=a$ of order $n_1$ we can calculate its residue using the well known identity 
\begin{equation*}
    0=\beta(a)=\eval{Res}_{z=a}gdh=\frac{1}{(n_1-1)!}\eval{\left(\frac{d}{dz}\right)^{n_1-1}}_{z=a}(z-a)^{n_1}gdh.
\end{equation*}
Taking the real part of this equation and using the fact that $b\in i\bR^*$ we obtain
\begin{equation*}
    0=\frac{\pm2}{(n_1-1)!}\eval{\left(\frac{d}{dz}\right)^{n_1-1}}_{z=a}\sum_{k=0}^{n_1-2}\binom{n_1-2}{k}a^kz^{n_0-1+k}.
\end{equation*}
Since
\begin{equation*}
    \eval{\left(\frac{d}{dz}\right)^{n_1-1}}_{z=a}z^{n_0-1+k}=\begin{cases}
        0,\quad \text{if}\quad k<n_1-n_0,\\ \frac{(n_0-1+k)!}{(n_0-n_1+k)!}a^{n_0-n_1+k},\quad \text{if}\quad n_1-n_0\leq k,
    \end{cases}
\end{equation*}
we have that
\begin{equation}\label{sumation}
    0=\sum_{\max\{(n_1-n_0),0\}}^{n_1-2}\binom{n_1-2}{k}\frac{(n_0-1+k)!}{(n_0-n_1+k)!}a^{n_1-n_0+2k}.
\end{equation}
Notice that this sum has always at least one term independently on the choices of $n_0,n_1\geq 2$. Also observe that since $n_0+n_1=m+1$ with $m$ odd, we have that $n_1-n_0=n_1+n_0-2n_0$ has to be even. This implies that all the powers of $a$ in the sum \eqref{sumation} are even. Consequently all the terms in the sum \eqref{sumation} are strictly positive and we arrive at a contradiction. Therefore there are no complete minimal immersions of twice-punctured projective spaces with parallel ends and closed curvature lines around its ends.
\end{proof}
\begin{cor}
    The Zhang's non-orientable minimal immersions of twice-punctured projective spaces in $\bR^3$ do not belong to the class $\cC$.
\end{cor}

We are able to produce other one parameter families of complete minimal immersions deforming the Oliveira's surface applying a symmetry principle in order to solve the period problem.
\begin{thm}\label{CarlosTFirst}
   There exist $r,R\in \bR^{+}$ such that for all $\lambda\in (0,r)\cup(R,+\infty)$  the Weierstrass data given by
\begin{equation*}
  M=\bS^2\setminus \{0,\infty,1,-1\},\quad g_{\lambda}(z)=\frac{\left(\overline{p_{\lambda}}^3z^3+1\right)\left(-\overline{q}_{\lambda}^2z^2+1\right)}{\left(z^3-p_{\lambda}^3\right)\left(z^2-q_{\lambda}^2\right)} \quad \eta_{\lambda}=i\frac{\left(z^3-p_{\lambda}^3\right)^2\left(z^2-q_{\lambda}^2\right)^2}{z^2(z^2-1)^4}dz,
\end{equation*}
where
\begin{gather*}
     p_{\lambda}^3=-i\lambda,\quad  q_{\lambda}^2=\frac{-10(1+\lambda^2)\pm\sqrt{P(\lambda)}}{2(1-35\lambda^2)}\in \bR,\quad P(\lambda)=240\lambda^4-4704\lambda^2+240. 
\end{gather*}
defines two families of complete minimal immersions $X'^{\pm}_{\lambda}:\mathbb{RP}^2\setminus\{[0],[1]\}\to \bR^3$ of total curvature $-10\pi$ with an embedded planar end at $[0]$ and an immersed end of Enneper type at $[1]$, whose asymptotic planes are orthogonal to each other. Up to equivalence, $X'^{\pm}_{0,\infty}$ are the Oliveira's two-ended Möbius bands. Moreover the lines of curvature of $X'^{\pm}_{\lambda}$ are not closed around any of the two ends.
\end{thm}

\begin{figure}
     \centering         \includegraphics[width=0.4\textwidth]{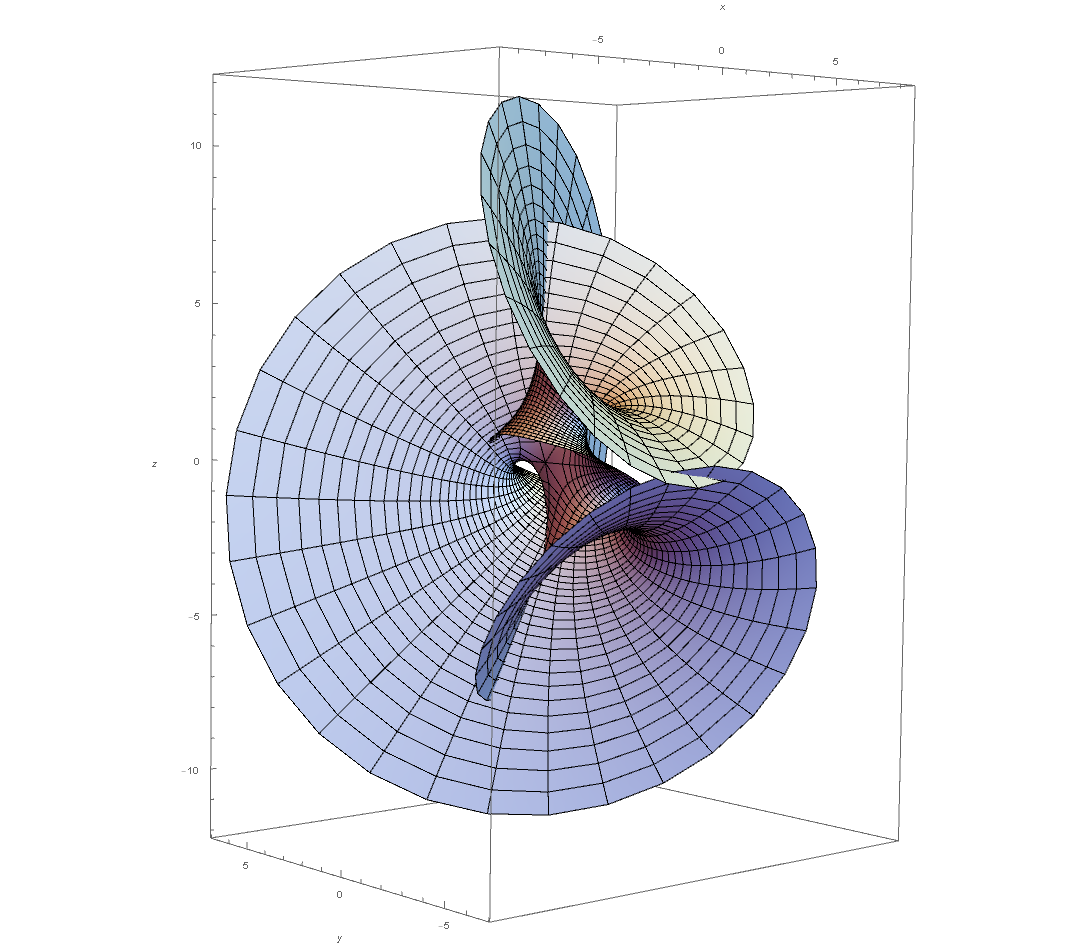}
         \caption{Surface from Theorem \ref{CarlosTFirst}}
         \label{Carlos T first surface}
     \end{figure} 
\begin{proof}

We start by fixing $a=1$ and $k_0=k_1=0$ and $(n_1,n_2)=(2,4)$ in equation \eqref{Weierstrasstwopunctures}. Therefore $m=5$ and $m-k_0-k_1=5$ so we have 
\begin{gather*}
    g=\prod_{j=1}^{5}\frac{(\overline{b_j}z+1)}{(z-b_j)}, \quad \eta=i\frac{\prod_{j=1}^{5}(z-b_j)^2}{z^{2}(z^2-1)^{4}}dz.
\end{gather*}
In order to solve the period problems, suppose that we divide the points $\{b_j\}_{j=1}^5$ into two groups. We choose $b_1,b_2,b_3$ arranged in a symmetrical way as the vertices of an equilateral triangle 
\begin{equation*}
    b_1=p,\quad b_2=pe^{\frac{2\pi i}{3}},\quad b_3=pe^{\frac{4\pi i}{3}},\quad p\in \bC \setminus \{-1,0,1\}.
\end{equation*}
Whereas we choose
\begin{equation*}
    b_4=q,\quad b_5=-q,\quad  q\in \bC \setminus \{-1,0,1\}.
\end{equation*}
Notice that the surface will have a horizontal tangent plane exactly at the points $\{b_j,\cJ(b_j)\}_{j=1}^5$.

The choice of $b_1,b_2,b_3$ implies
\begin{equation*}
    (z-b_1)(z-b_2)(z-b_3)=z^3-p^3,\quad (\overline{b_1}z+1)(\overline{b_2}z+1)(\overline{b_3}z+1)=\overline{p}^3z^3+1,
\end{equation*}
and the choice of $b_4,b_5$ implies
\begin{equation*}
    (z-b_4)(z-b_5)=z^2-q^2,\quad (\overline{b_4}z+1)(\overline{b_5}z+1)=-\overline{q}^2z^2+1.
\end{equation*}
Therefore
\begin{equation*}
   g(z)=\frac{\left(\overline{p}^3z^3+1\right)\left(-\overline{q}^2z^2+1\right)}{\left(z^3-p^3\right)\left(z^2-q^2\right)} \quad \eta=i\frac{\left(z^3-p^3\right)^2\left(z^2-q^2\right)^2}{z^2(z^2-1)^4}dz.
\end{equation*}
After expanding out the terms and using the residue calculation tool \cite[Lemma 3.1 and Lemma 3.2]{KatoHamadaVariousTypesEnds}, we obtain
\begin{equation*}
    \alpha(0)=0,\quad \beta(0)=0,\quad \gamma(0)=0,
\end{equation*}
and 
\begin{equation*}
    \alpha(1)=-\frac{i}{32}\left(35p^6q^4-10p^6q^2-p^6+q^4+10q^2-35\right),
\end{equation*}
\begin{equation*}
    \beta(1)=\frac{i}{32}\left(35\overline{p}^6\overline{q}^4-10\overline{p}^6\overline{q}^2-\overline{p}^6+\overline{q}^4+10\overline{q}^2-35\right),
\end{equation*}
\begin{equation*}
    \gamma(1)=\frac{i}{32}\left(-35p^3q^2+5p^3(1+|q|^4)+p^3\overline{q}^2+\overline{p}^3q^2+5\overline{p}^3(1+|q|^4)-35\overline{p}^3\overline{q}^2\right).
\end{equation*}
Clearly the conditions $\gamma(0)\in \bR$ and $\beta(0)+\overline{\alpha}(0)=0$ are satisfied. Accordingly with Lemma \ref{simplifiedresidueproblem} the period problem is solved if and only if $\gamma(1)\in \bR$ and $\beta(1)+\overline{\alpha}(1)=0$. We then have the equations
\begin{gather}\label{residueequationsCarlosFirst}
    (35p^6+1)q^4+10(1-p^6)q^2-(35+p^6)=0,\quad Re\left(p^3\left(\overline{q}^2+5(1+|q|^4)-35q^2\right)\right)=0.
\end{gather}
Let us describe one solution to this system of equations. Set
\begin{equation*}
    p^3=-i\lambda,
\end{equation*}
and consider the polynomial
\begin{equation*}
    P(\lambda)=240\lambda^4-4704\lambda^2+240,
\end{equation*}

This polynomial is even and has four distinct real roots at the points $-R<-r<0<r<R$ with $R=\frac{1}{r}$. For $\lambda\in (0,r)\cup (R,+\infty)$ we then have $P(\lambda)\geq 0$. Therefore for this choice of $\lambda$ we have by the first equation of \eqref{residueequationsCarlosFirst} that
\begin{equation*}
    q^2=\frac{-10(1+\lambda^2)\pm\sqrt{P(\lambda)}}{2(1-35\lambda^2)}\in \bR,
\end{equation*}
and we see that the second equation \eqref{residueequationsCarlosFirst} is also satisfied. Finally when $\lambda\to 0$ we have
     \begin{equation*}
        p\to 0,\quad  q^2\to \frac{-10\pm \sqrt{240}}{2},\quad g\to \frac{-q^2z^2+1}{z^3(z^2-q^2)},\quad \eta\to i\frac{z^6(z^2-q^2)^2}{z^2(z^2-1)^4}dz,
     \end{equation*}
     which is up to reparametrization, rigid motion and homothety the Oliveira's two ended Möbius band. Similarly when $\lambda\to +\infty$ we have
     \begin{equation*}
         p\to -i\infty,\quad \frac{1}{q^2}\to \frac{-10\mp \sqrt{240}}{2},\quad g\to -\frac{z^3(z^2-\frac{1}{q^2})}{\frac{z^2}{q^2}-1},\quad \eta\to (-\lambda^2q_{\lambda}^4)i\frac{\left(\frac{z^2}{q^2}-1\right)^2}{z^2(z^2-1)^4} ,
     \end{equation*}
     which is up to a scaling factor $\lambda^2q^4\in \bR^{+}$ and a reflection with respect to the origin the Oliveira's surfaces.
     
     The choice $(n_1,n_2)=(2,4)$ in the Weierstrass data implies by Theorem \ref{embeddedends} and   Corollary \ref{jorgemeeksmultiplicitiespunctureprojectivespaces} that the end corresponding to $[0]$ is embedded and $[1]$ is immersed of Enneper type. Since $\gamma(0)=0$, which is the residue of the height differential at $z=0$, we conclude that the embedded end is planar by Remark \ref{catenoidalvsplanar}. The asymptotic behavior of the Gauss map at the ends is
     \begin{equation*}
         \lim_{z\to0}g_{\lambda}(z)=\frac{1}{p_{\lambda}^3q_{\lambda}^2}\in i\bR,\quad \lim_{z\to 1} g_{\lambda}(z)=\frac{(1+\overline{p}_{\lambda}^3)(1-\overline{q}_{\lambda}^2)}{(1-p_{\lambda}^3)(1-q_{\lambda}^2)}=1\in \bR,
     \end{equation*}
     which implies that the asymptotic planes to the ends are orthogonal to each other. Finally, the planar end $[0]$ does not admit closed curvature lines by Theorem \ref{PlanarEndsDoNotAdmitClosedLines}. An analysis of the behavior of $q_\lambda^2$ and a calculation of the derivative of the Gauss map at the end $[1]$ shows that
     \begin{equation*}
        q_{\lambda}^2\in \bR\setminus \{\pm 1\},\quad  g'_{\lambda}(1)=\frac{(q_{\lambda}^2-5)-i\lambda(5q_{\lambda}^2-1) }{(1-q_{\lambda}^2)(1-p_{\lambda}^3)}\neq 0,\quad 
     \end{equation*}
     which implies that $r_{[1]}\mathfrak{g}=0$. Since the Jorge-Meeks multiplicity of the end $[1]$ is equal to $3$, we have by Theorem \ref{JorgeMeeksHopf}, that the Hopf differential has a pole of order $4$ at the end $[1]$, so by Theorem \ref{dynamicslinescurvature} the lines of curvature spiral at this end.  
\end{proof}
\begin{rem}
\normalfont{
     It is interesting to notice that $R=\frac{1}{r}$ and that when $\lambda\to r$ or $\lambda\to R$, since $P(\lambda)\to 0$, we see that the surfaces $X^{\pm}_r,X^{\pm}_{R}$ coincide. We can also see that starting at one Oliveira's surface at $\lambda\to 0^{+}$ we go to $X_r$ and since this surface is $X_R$ we can continue the deformation for $\lambda$ from $R$ to infinity arriving at the other Oliveira's surface. This same behavior is observed in the Toubiana family, so it could be that the two constructions are equivalent.}
\end{rem}

Now we prove Theorem \ref{secondsurface}, which produces still another deformation of Oliveira's surface with a single closed curvature line around exactly one of the two ends:

\begin{thm}\label{CarlosTSecond}
The Weierstrass data given by
\begin{gather*}
M=\bS^2\setminus \{0,\infty,1,-1\},\quad g_{\lambda}(z)=\frac{z(1-\overline{p}_{\lambda}^2z^2)(1-\overline{q}_{\lambda}^2z^2)}{(z^2-p_{\lambda}^2)(z^2-q_{\lambda}^2)},\quad     \eta_{\lambda}=i\frac{(z^2-p_{\lambda}^2)^2(z^2-q_{\lambda}^2)^2}{z^{2}(z^2-1)^{4}}dz ,
\end{gather*}
with
\begin{gather*}
q_{\lambda}^2=\frac{i\lambda}{p_{\lambda}^2},\quad   Q(\lambda)=\frac{-(2+10i\lambda)\pm\sqrt{(2+10i\lambda)^2-4(35\lambda^2+2i\lambda+5)}}{2},
\end{gather*}
and
\begin{gather*}
p_{\lambda}^2=\frac{Q(\lambda)+ \sqrt{Q(\lambda)^2-4i\lambda}}{2}\ \text{for $X^{\pm}_{\lambda}$},\quad p_{\lambda}^2=\frac{Q(\lambda)-\sqrt{Q(\lambda)^2-4i\lambda}}{2}\ \text{for $Y^{\pm}_{\lambda}$}.
\end{gather*}
defines four families of complete minimal immersions $X'^{\pm}_{\lambda},Y'^{\pm}_{\lambda}:\mathbb{RP}^2\setminus\{[0],[1]\}\to \bR^3$ of total curvature $-10\pi$ with an embedded catenoidal end at $[0]$ foliated by closed curvature lines, and an immersed end of Enneper type at $[1]$, whose asymptotic planes are orthogonal to each other. Moreover the surfaces $X'^-_{\lambda}$ and $Y'^+_{\lambda}$ converge up to equivalence to the two Oliveira's surfaces respectively.
\end{thm}

\begin{figure}
     \centering
     \begin{subfigure}[b]{0.25\textwidth}
         \centering
         \includegraphics[width=\textwidth]{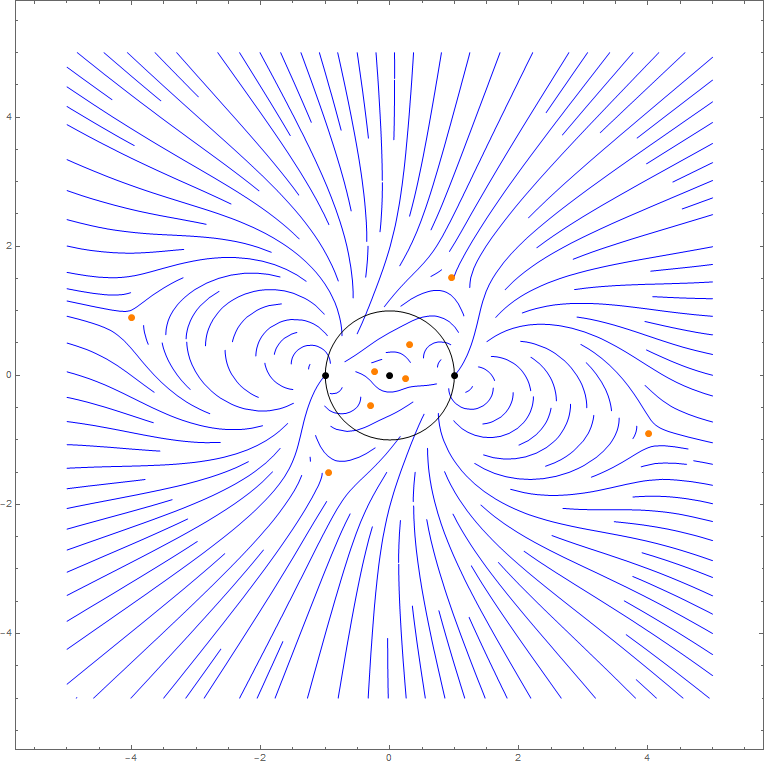}
         \caption{General view at a large scale. }
         \label{Carlos T second surface curvature lines}
     \end{subfigure}
     \hfill
     \begin{subfigure}[b]{0.38\textwidth}
         \centering
         \includegraphics[width=\textwidth]{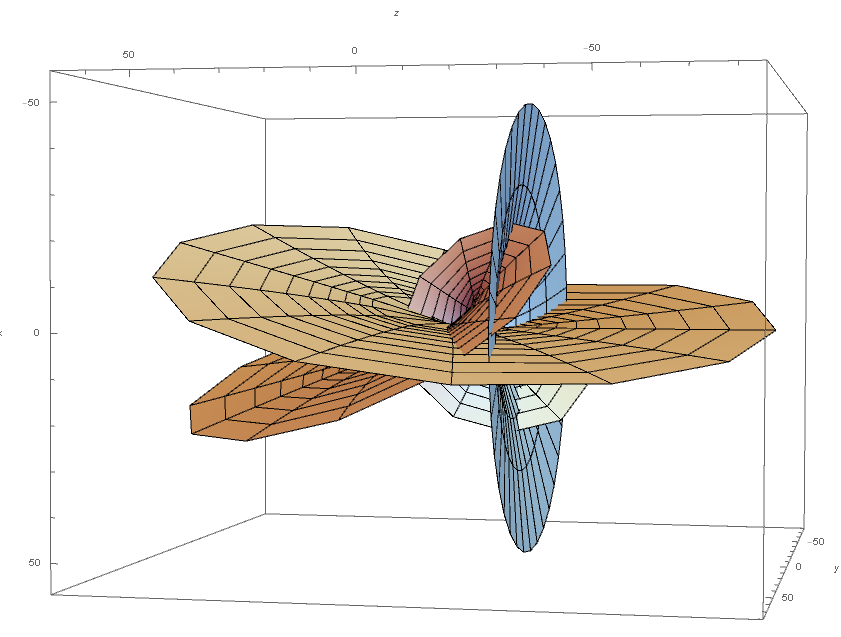}
         \caption{The Surface from Theorem \ref{Carlos T second surface}. }
         \label{Carlos T second surface}
     \end{subfigure}
     \hfill
     \begin{subfigure}[b]{0.25\textwidth}
         \centering
\includegraphics[width=\textwidth]{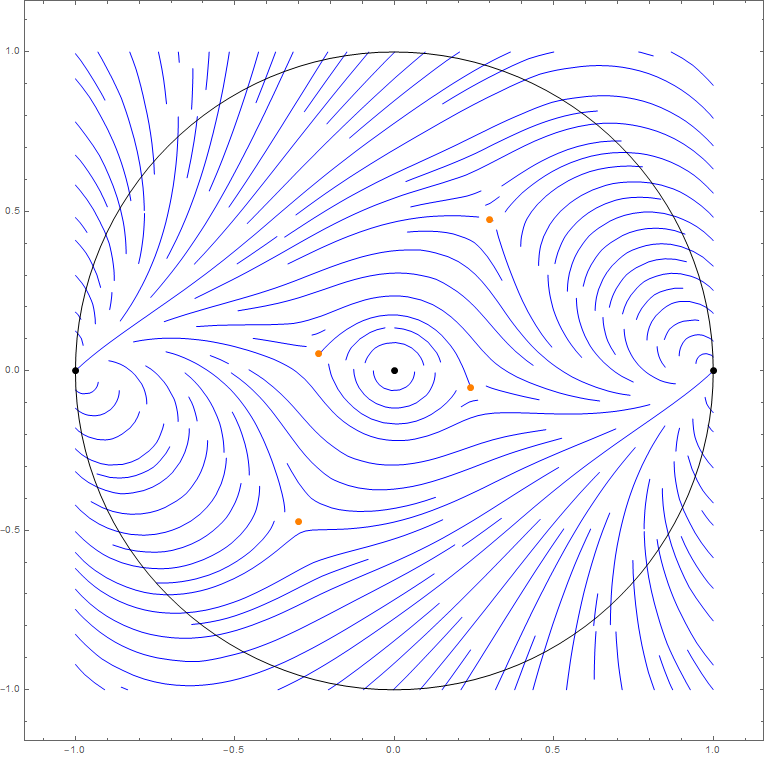}
         \caption{A closer look at the closed curvature line.}
         \label{fig:closer look Carlos T second surface curvature lines}
     \end{subfigure}
     \caption{The surface of Theorem \ref{CarlosTSecond} and its maximal principal foliation.}
     \end{figure}
     
\begin{proof}
The period condition is verified using the same strategy as in the proof of Theorem \ref{CarlosTFirst}. We next show using Corollary \ref{CriterioclassC} that the lines of curvature around the end corresponding to $z=0$ are indeed closed. On one hand $\frac{dg}{g}$ has a simple pole with residue $1$ at $z=0$ and on the other hand the height differential
    \begin{equation*}
        dh=i\frac{(1-\overline{p}^2z^2)(1-\overline{q}^2z^2)(z^2-p^2)(z^2-q^2)}{z(z^2-1)^4},
    \end{equation*}
    has also a simple pole at $z=0$ with residue $ip^2q^2=-\lambda\in \bR^*$, which implies that the Hopf differential has a pole of order two at $z=0$ with
\begin{equation*}
    \lim_{z\to 0}z^2\phi(z)=-\frac{ip^2q^2}{2}=\frac{\lambda}{2}\in \bR^{*},
\end{equation*}
and therefore we have a principal cycle around the embedded end $z=0$ which is catenoidal since the residue of the height differential does not vanish at that point. Finally we see that
\begin{equation*}
   \lim_{\lambda\to +\infty}\frac{Q(\lambda)}{i\lambda}=\frac{-10\pm \sqrt{240}}{2},
\end{equation*}
which implies for $X_{\lambda}^-$
\begin{equation*}
    \lim_{\lambda\to +\infty}\frac{1}{p_{\lambda}^2}=\frac{-10-\sqrt{240}}{2},\quad \lim_{\lambda\to +\infty}\frac{1}{q_{\lambda}^2}=0,
\end{equation*}
and for $Y_{\lambda}^+$
\begin{equation*}
    \lim_{\lambda\to +\infty}\frac{1}{p_{\lambda}^2}=\frac{-10+\sqrt{240}}{2},\quad \lim_{\lambda\to +\infty}\frac{1}{q_{\lambda}^2}=0.
\end{equation*}
After an appropriate scaling at each $\lambda$ and a reflection through the origin, the surfaces $X_{\lambda}^-,Y_{\lambda}^+$ converge to the Oliveira's two-ended Möbius bands respectively for large $\lambda\to +\infty$. The asymptotic planes to the ends are orthogonal since $g_{\lambda}(0)=0$ and $\abs{g_{\lambda}(1)}=1$.
\end{proof}
\begin{rem}
\normalfont{
    A numerical simulation suggest that the ramification order of the Gauss map at the end $[1]$ is zero, which would imply by Theorem \ref{JorgeMeeksHopf} that the Hopf differential has a pole of order $4$ at $[1]$, and therefore the lines of curvature spiral around that end, according to Theorem \ref{dynamicslinescurvature}.}
\end{rem}

The following theorem studies the space of quadratic differential on the fourth punctured sphere with closed curvature lines.
\begin{figure}
    \centering    \includegraphics[width=0.35\linewidth]{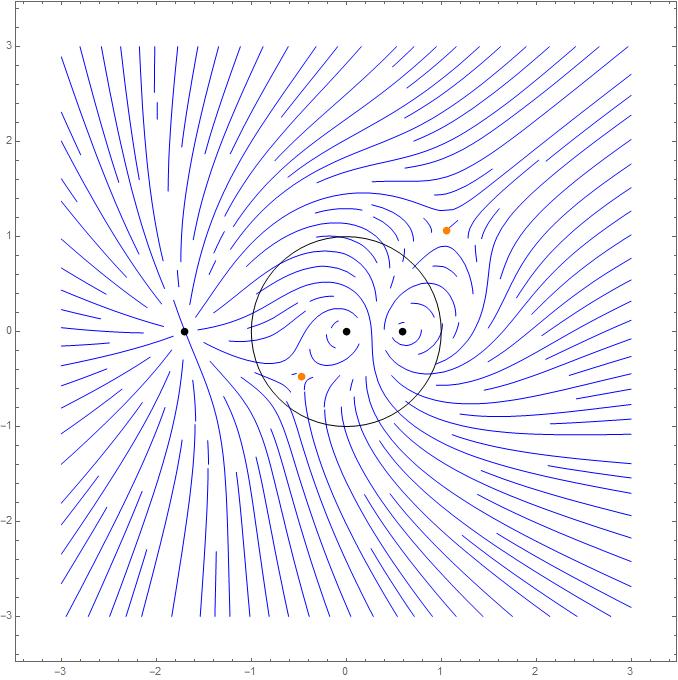}
    \caption{Quadratic differential $Q_{*}$ of Example \ref{Spaceofquadraticdifferentialswithclosedcurvaturelines} with closed curvature lines.}
    \label{fig:quadratic differential with closed curvature lines}
\end{figure}
     \begin{thm}
         Let $\cQ$ be the space of meromorphic quadratic differentials on the fourth punctured sphere $\bS^2\setminus\{0,\infty,a,-\frac{1}{a}\}$, $a\in \bR^{*}$ with possible poles only at the punctures, having the transformational property of Proposition \ref{transformationHopfgeneralized} and closed curvature lines around the punctures. Then 
         \begin{equation*}
    \cQ=\left\lbrace i\lambda \frac{(z-\alpha_1)(z-\alpha_2)(\overline{\alpha_1}z+1)(\overline{\alpha_2}z+1)}{z^2(z-a)^2(az+1)^2}dz^2\mid  \lambda\in \bR,\  \alpha_1\alpha_2,(a-\alpha_1)(a-\alpha_2)(a\overline{\alpha_1}+1) (a\overline{\alpha_2}+1)\in i\bR^{*}\right\rbrace.
\end{equation*}
     \end{thm}
     \begin{proof}
       Let $Q\in \cQ$. Since $Q$ has poles of order two at the punctures and since it is meromorphic, it must be of the form
    \begin{equation*}
        Q=\frac{P(z)}{z^2(z-a)^2(az+1)^2}dz^2,
    \end{equation*}
    where $P$ is a polynomial with $P(0),P(a),P\left(-\frac{1}{a}\right)\neq 0$. By the Poincaré-Hopf theorem applied to the maximal principal foliation associated with this quadratic differential we obtain
    \begin{equation*}
        2=\cX(\bS^2)=4-\sum_{\{\alpha,P(\alpha)=0\}}\frac{n_{\alpha}}{2},
    \end{equation*}
    where $n_{\alpha}$ is the multiplicity of the zero $\alpha$ of $P$. Therefore we must have that $\sum_{\{\alpha,P(\alpha)=0\}}n_{\alpha}=4$
     which implies that $\deg(P)=4$. We can therefore write
    \begin{equation*}
        Q=C\frac{(z-\alpha_1)(z-\alpha_2)(z-\alpha_3)(z-\alpha_4)}{z^2(z-a)^2(az+1)^2}dz^2,\quad \alpha_j\neq 0,a,-\frac{1}{a}.
    \end{equation*}
    The transformational property of Proposition \ref{transformationHopfgeneralized} implies that
    \begin{equation*}
        C\prod_{j=1}^4\alpha_j=-\overline{C},\quad \{\alpha_j\}_{j=1}^4=\left\lbrace-\frac{1}{\overline{\alpha_k}}\right\rbrace_{k=1}^4.
    \end{equation*}
    Since $\alpha_j\neq -\frac{1}{\overline{\alpha_j}}$ it must be the case that up to a permutation $\alpha_3=-\frac{1}{\overline{\alpha_1}}$, $\alpha_4=-\frac{1}{\overline{\alpha_2}}$ so we can write
\begin{equation*}
    Q=\frac{C}{\overline{\alpha_1\alpha_2}}\frac{(z-\alpha_1)(z-\alpha_2)(\overline{\alpha_1}z+1)(\overline{\alpha_2}z+1)}{z^2(z-a)^2(az+1)^2}dz^2,\quad \frac{C}{\overline{\alpha_1\alpha_2}}=-\overline{\left(\frac{C}{\overline{\alpha_1\alpha_2}}\right)}.
\end{equation*}
This implies $\frac{C}{\overline{\alpha_1\alpha_2}}=i\lambda$ with $\lambda\in \bR$. In order to satisfy the closed curvature condition it is sufficient to impose the limit condition on $Q$ at the punctures $z=0$ and $z=a$. Therefore we conclude that the space $\cQ$ is described as
\begin{equation*}
    \cQ=\left\lbrace i\lambda \frac{(z-\alpha_1)(z-\alpha_2)(\overline{\alpha_1}z+1)(\overline{\alpha_2}z+1)}{z^2(z-a)^2(az+1)^2}dz^2\mid \lambda\in \bR,\  \alpha_1\alpha_2,(a-\alpha_1)(a-\alpha_2)(a\overline{\alpha_1}+1) (a\overline{\alpha_2}+1)\in i\bR^{*}\right\rbrace.
\end{equation*}
     \end{proof}
      
\begin{exam}\label{Spaceofquadraticdifferentialswithclosedcurvaturelines}
\normalfont{The space $\cQ$ is non-empty, for instance the following quadratic differential whose maximal principal foliation is depicted in Figure \ref{fig:quadratic differential with closed curvature lines} belongs to the space $\cQ$:}
\begin{equation*}
    Q_{*}=i\frac{(z-\alpha)\left(\overline{\alpha}z+1\right)^2}{z^2(z-a)^2\left(az+1\right)^2}dz^2,\quad \alpha=te^{i\frac{\pi}{4}},\quad a\in \bR^{*},\ 2t(1-t^2)a^2+\sqrt{2}(t^4-4t^2+1)a-2t(1-t^2)=0.
\end{equation*}
\end{exam}

Notice that the Hopf differential of a non-orientable minimal twice-punctured projective space with finite total curvature and closed curvature lines around the ends must belong to this space. Nonetheless it is still a problem to determine which of these quadratic differentials can be realized as the Hopf differential of such a non-orientable minimal immersion.

It is possible to check that the Hopf differential of the Kato-Hamada surfaces have a pole of order two at each end, and moreover there is sufficiently freedom in the space of parameters in order to find closed curvature lines around their ends. This suggests that the Kato-Hamada surfaces are potentially good models of non-orientable free boundary minimal surfaces in the unit ball $\bB^3.$

\printbibliography

\end{document}